\documentclass[12pt]{article}
\usepackage[margin=1.25in]{geometry}
\usepackage[english]{babel}
\usepackage{etoolbox}
\usepackage{textcomp}
\usepackage{amsmath}
\usepackage{amsthm}
\usepackage{thmtools}
\usepackage{amssymb}
\usepackage{bbm}
\usepackage{graphicx}
\usepackage{tikz}
\usepackage{mathtools}
\usepackage[shortlabels]{enumitem}
\usepackage{hyperref}
\usetikzlibrary{cd}
\usepackage{cleveref}

\makeatletter

\def\env@sqcases{%
  \let\@ifnextchar\new@ifnextchar
  \left\lbrack
  \def\arraystretch{1.2}%
  \array{@{}l@{\quad}l@{}}%
}
\makeatother

\hypersetup{
    colorlinks=true,
    linkcolor=blue,
    filecolor=magenta,
    pdftitle={Zhizhin 2nd Year Paper},
    }
    
\newcommand{\R}{\ensuremath{\mathbb R}}
\newcommand{\C}{\ensuremath{\mathbb C}}
\newcommand{\Z}{\ensuremath{\mathbb Z}}

\newcommand{\A}{\ensuremath{\mathbb A}}

\newcommand{\defeq}{\vcentcolon=}

\renewcommand{\O}{\ensuremath{\mathcal{O}}}
\renewcommand{\div}{\operatorname{div}}
\renewcommand{\Im}{\operatorname{Im}}

\renewcommand{\d}{\ensuremath{\mathfrak d}}
\DeclareMathOperator{\Div}{Div}
\DeclareMathOperator{\CaDiv}{CaDiv}

\DeclareMathOperator{\rk}{rk}

\DeclareMathOperator{\Spec}{Spec}

\DeclareMathOperator{\Supp}{Supp}

\DeclareMathOperator{\Vol}{Vol}
\DeclareMathOperator{\MVol}{MVol}
\DeclareMathOperator{\Ker}{Ker}
\DeclareMathOperator{\Hom}{Hom}
\DeclareMathOperator{\Conv}{Conv}

\DeclareMathOperator{\Pic}{Pic}

\newcommand{\citestacks}[1]{\cite[\href{https://stacks.math.columbia.edu/tag/#1}{Tag #1}]{stacks-project}}
            
\theoremstyle{definition}

\newtheorem{theorem}{\textbf{Theorem}}[section]
\newtheorem{example}[theorem]{\textbf{\sc{Example}}}
\newtheorem*{example*}{\textbf{\sc{Example}}}
\newtheorem{claim}[theorem]{\textbf{\sc{Claim}}}
\crefname{claim}{Claim}{Claims}
\newtheorem*{claim*}{\textbf{\sc{Claim}}}
\crefname{claim*}{Claim}{Claims}

\crefname{notation}{Notation}{Notations}
\newtheorem{lemma}[theorem]{\textbf{\sc{Lemma}}}
\crefname{lemma}{Lemma}{Lemmas}
\newtheorem{remark}[theorem]{\textbf{\sc{Remark}}}
\newtheorem{corollary}[theorem]{\textbf{\sc{Corollary}}}
\crefname{corollary}{Corollary}{Corollaries}

\theoremstyle{plain}
\newtheorem{definition}[theorem]{Definition}

\usepackage{sectsty}
\sectionfont{\centering\sc}
\makeatletter
\def\@seccntformat#1{\csname the#1\endcsname.\ }
\makeatother

\makeatletter
\let\@fnsymbol\@arabic
\makeatother

\addto\extrasenglish{

}

\title{Irreducible components of Toric Complete Intersections}
\author{Andrey Zhizhin}
\date{August 2025}

\begin{document}

\maketitle
\begin{abstract}
    An equivariant linear system on a toric variety is a linear system invariant under the torus action. We study the number of irreducible components of the complete intersection of general divisors from a fixed collection of equivariant linear system on a toric variety $X$. An explicit formula for the number of components was obtained in \cite{Khovanskii2016} for the case $X = T^n$ over $\C$ and generalized to an algebraically closed field of arbitrary characteristic in \cite{Zhizhin2024}. Building on these results, we give a recursive formula for an arbitrary toric variety.
\end{abstract}

\tableofcontents

\section{Introduction}
\subsection{Landscape}

\paragraph{Motivation} 
One could consider the following natural framework. We fix an ambient space such as $\A^n$ or $\mathbb P^n$ and study varieties defined by general "equations" of a fixed form -- that is, equiations with fixed monomials (e.g. homogeneous equations and homogeneous monomials in the case of $\mathbb P^n$). As usual, by an "equation" on a general variety we mean a Cartier divisor. The corresponding generalization of a space of systems with fixed monomials is an equivariant linear system on a toric variety $X$ --- we will explain the analogy in more details below. The aspect we are concerned with is the number of irreducible components of the variety defined by such general "system of equations". Once we set our focus to toric varieties the problem splits into two parts. First, we need to solve the problem (i.e. find the number of irreducible components) in the case when our variety $X$ is just the algebraic torus $T^n$. Then we cut our variety $X$ into tori and analyze which components from these tori contribute to the number of components of the global intersection. The former part was resolved over $\C$ in \cite{Khovanskii2016} and later extended to fields of arbitrary characteristic in \cite{Zhizhin2024}. The present note treats the latter part.

\paragraph{Main result} The main difficulty here lies not in the proofs but in the language one needs to get through in order to set up the result. In the following few paragraphs we will try to walk the reader through all the necessary steps as quickly as possible. We begin with the standard notation. Let $N \simeq \Z^n$ be a lattice, $M \defeq \Hom(M, \Z)$ be the dual lattice, $\Sigma$ be a fan\footnote{for the purposes of this text it is convenient to consider only the integral points of all the cones, i.e. we define a cone as a saturated subset of the lattice that contains zero and is closed under addition} in $N$ and $X_\Sigma$ be the corresponding toric variety. 

Let $D$ be an equivariant Cartier divisor on $X_\Sigma$. Then we define\footnote{we denote by $|\Sigma|$ the support of $\Sigma$, i.e. the union of its cones} $\psi_D : |\Sigma| \to \Z$, $\lambda \mapsto \deg \lambda^*D$, i.e. for any $\lambda: T^1 \to T^n$ we fix the unique extension $\bar \lambda : \A^1 \to X_\Sigma$, then $\bar \lambda^*D$ is a divisor on $\A^1$ that is supported at most at zero: $\lambda^*D = m \cdot \{ 0 \}$, and we define $\psi_D(\lambda) \defeq m$. We call $\psi_D$ \textit{the support function}\footnote{in some sources (e.g. in \cite{CoxLittleSchenk}) the support function has the opposite sign} \textit{of $D$}. We will need the following fact: let $\sigma \in \Sigma$ be a cone and $\O_\sigma \subset X_\Sigma$ be the corresponding orbit, then $\O_\sigma \not \subset \Supp D$ if and only if $\psi_D(\sigma) = 0$. Also the support function is linear on cones, $\psi_{D + E} = \psi_D + \psi_E$ for any equivariant Cartier divisors $D, E$, and an equivariant Cartier divisors is completely defined by its support function --- we recall all the necessary material in \autoref{linsys:ssec:equivar_div}.

An equivariant linear system on $X_\Sigma$ is a linear system $\d$ such that for any divisor $D \in \d$ and any $t \in T^n$ we have that $t^*D \in \d$. Linear systems correspond to pairs $(\mathcal L, V)$, where $\mathcal L$ is a line bundle and $V \subset \Gamma(X_\Sigma, \mathcal L)$ is a non-zero subspace of the space of its sections. Any equivariant linear system $\d$ contains an equivariant Cartier divisor $D \in \d$, so we may assume that $\mathcal L = \O(D)$ --- in particular, now\footnote{For a set $S$ we denote by $k \cdot S$ the $k$-vector space spanned by $S$} $V \subset~\Gamma(T^n, \O(D)) =~k \cdot M$ as $D|_{T^n} = 0$. Consider the set $A \defeq V \cap M$. Since $\d$ is equivariant, $V = k \cdot A$. The combinatorial datum $(A, \psi \defeq \psi_{D_0})$, where $D_0 \in \d$ is an equivariant divisor, determines an equivariant linear system uniquely and henceforth we will work only with the combinatorial data. We say that $\d$ degenerates on $\O_\sigma$ if $\O_\sigma$ lies in the base locus of $\d$, i.e. $\O_\sigma \subset D$ for all $D \in \d$. The following combinatorial criterion applies: $\d$ degenerates on $\O_\sigma$ if and only if for any $\chi \in A$ there is $\lambda \in \sigma$ such that $\langle \chi, \lambda \rangle + \psi(\lambda) > 0$. If $\d$ does not degenerate on $\O_\sigma$, then $\d|_{\O_\sigma}$ is defined by the combinatorial data $(\chi^{-1} \cdot A \cap \sigma^\bot, 0)$, where $(\psi + \langle\chi, - \rangle)|_\sigma \equiv 0$. We will denote by $A^\sigma \defeq \chi^{-1} \cdot A \cap \sigma^\bot$ if $\d$ does not degenerate on $\O_\sigma$ --- otherwise we put $A^\sigma \defeq \varnothing$. Proofs for all the mentioned statements can be found in \autoref{linsys:ssec:equivar_ls}.

The last ingredient is the Khovanskii theorems. Given finite subsets $A_1, \dots, A_m \subset~M$ and the corresponding linear systems $\d_1, \dots, \d_m$ on $T^n$ we denote by $K_{T^n}(A_1, \dots, A_m)$ the number of irreducible components of $D_1 \cap \dots \cap D_m$ for the general $D_1 \in \d_1, \dots,$ $D_m \in \d_m$ --- the number is computed in \cite{Khovanskii2016} (and \cite{Zhizhin2024}) --- we placed the algorithm for its computation in \autoref{prelim:ssec:khovanskii}.

Finally, we are ready to formulate our result. Let $\d_1, \dots, \d_m$ be equivariant linear systems on $X_\Sigma$ defined by the combinatorial data $(A_1, \psi_1), \dots, (A_m, \psi_m)$. Consider the sets $\mathcal D(\sigma) \defeq \{ i\ |\ \d_i \text{ degenerates on } \O_\sigma \}$ and the function $d(\sigma) \defeq |\mathcal D(\sigma)| - \dim \sigma$. Define the set:
\[
    \mathcal S = \{ \sigma \in \Sigma \ |\ d(\sigma) \ge d(\tau) \text{ for all faces } \tau \text{ of } \sigma\}.
\]
Then for the general $D_1 \in \mathfrak d_1, \dots, D_m \in \mathfrak d_m$ the intersection $D_1 \cap \dots \cap D_m$ has the following number of irreducible components:
\[
\sum_{\sigma \in \mathcal S} K_{\O_\sigma}(A_1^\sigma, \dots, A_m^\sigma)
\]
where if $\mathfrak d_i$ degenerates along $\O_\sigma \xhookrightarrow{} X_\Sigma$, then we assume\footnote{i.e. we assume $K_\sigma(A_1^\sigma, \dots, A_m^\sigma) \defeq K_\sigma (A_{i_1}^\sigma, \dots, A_{i_l}^\sigma)$ for $\{ i_1, \dots, i_l \} \defeq \{1, \dots, m\} \backslash \mathcal D(\sigma)$} $ A_i^\sigma = \varnothing$.

\subsection{Paper Structure}
The opening \autoref{sec:preliminaries} consists of 3 independent parts. The first is \autoref{prelim:ssec:div_lin_sys} in which we recall the basic definitions and facts concerning divisors and linear systems on normal noetherian integral schemes. Then in \autoref{prelim:ssec:equivariance} we briefly discuss what we mean by equivariance in this text. The third part is \autoref{prelim:ssec:newton_polytopes} and \autoref{prelim:ssec:khovanskii}, where we recall all the Newton Polytope theory we need.

The following \autoref{sec:linsys} is the technical heart of the text. In the first two subsections (\autoref{linsys:ssec:pullbacks_div} and \autoref{linsys:ssec:pullbacks_ls}) we develop a general theory of pullbacks for divisors and linear systems --- we do not imply that the material there is by any means novel, but rather that it may lack a unified exposition in the literature. The following two subsections (\autoref{linsys:ssec:equivar_div} and \autoref{linsys:ssec:equivar_ls}) apply this theory in the equivariant toric setting, resulting in a combinatorial technique for restricting equivariant linear systems to closed invariant subvarieties (e.g., orbits).

Finally in \autoref{sec:counting} we use the results of \autoref{linsys:ssec:equivar_div} and \autoref{linsys:ssec:equivar_ls} to count the irreducible components. The section is short and follows naturally from the preceding discussion.

\section{Preliminaries}\label{sec:preliminaries}
\paragraph{Notation} We fix an algebraically closed base field $k$. We denote by $T^n$ the algebraic torus $\Spec k[x_1^{\pm 1}, \dots, x_n^{\pm 1}]$. As usual, $M$ stands for the character lattice $\Hom(T^n, T^1)$ and $N$ is the dual lattice $\Hom(T^1, T^n)$. $\langle \ ,\  \rangle : M \times N \to Z$ denotes the canonical pairing. For a fan $\Sigma$ in $N$ we denote by $X_\Sigma$ the corresponding toric variety and by $|\Sigma| \subset N$ the support of the fan, i.e. the union of all cones of $\Sigma$.

\subsection{Divisors \& Linear Systems}\label{prelim:ssec:div_lin_sys}
In this subsection we discuss the notions of Cartier divisors and Linear Systems. All the mentioned facts and definitions are standard and the reader may feel free to skip this subsection --- it serves as a reference source for \autoref{sec:linsys}.

\subsubsection{Divisors}
Throughout this subsection (and \autoref{sec:linsys}) we work only with integral normal noetherian schemes and by $X$ we always denote a non-empty integral normal noetherian scheme. By $\Div X$ we denote the free abelian group of formal $\Z$-linear combinations of points $\xi \in X$ such that $\dim \O_{X, \xi} = 1$. We call the elements of the group $\Div X$ Weil divisors or just divisors. Equivalently, we can say that $\Div X$ is the free abelian group of formal $\Z$-linear combinations of prime divisors, where a prime divisor on $X$ is an integral closed subscheme of codimension 1. We denote by $R(X)$ the field of rational functions of $X$ and by $\underline{R(X)}$ the $\O_X$-module of rational function. By Serre's criterion of normality for any $\xi \in X$ such that $\dim \O_{X, \xi} = 1$ the local ring $\O_{X, \xi}$ is a discrete valution ring and we denote that valuation by $v_\xi: R(X)^\times \twoheadrightarrow \Z $. Then for a non-zero function $f \in R(X)^\times$ we denote by $\div f$ the divisor $\sum_{\xi} v_\xi(f) \cdot \xi$ and call such divisors principal. For any non-empty open $U \subset X$ we have the natural homomorphism $\Div X \to \Div U, D \mapsto D|_U$ given by $\xi \mapsto \xi$ for $\xi \in U$ and $\xi \mapsto 0$ for $\xi \not \in U$. We say that $D$ is a Cartier Divisor if for any $p \in X$ there is an open neighbourhood $U \subset X$, $p \in U$ and a rational function $f \in R(X)^\times$ such that $D|_U =~\div f|_{U}$, i.e. Cartier divisors are locally principal divisors. As sum of two principal divisors is a principal divisor: $\div f + \div g = \div fg$, Cartier divisors form a subgroup of Weil divisors that we denote by $\CaDiv X \subset \Div X$. Given two divisors $D = \sum_\xi d_\xi \cdot \xi$, $E = \sum_\xi e_\xi \cdot \xi$ we write that $D \ge E$ if $d_\xi \ge e_\xi$ for all $\xi$. All of the following statements are well-known and we refer the reader to \cite[Ch. II, par. 6]{Hartshorne1977} for a complete treatise and most of the proofs.

\begin{definition}
    Let $D \in \Div X$ be a divisor on a normal noetherian integral scheme $X$. Then $\O_X(D)$ is the subsheaf of $\underline{R(X)}$ given by
    \[
    \Gamma(U, \O_X(D)) = \{f \in R(X)\ :\ (D + \div f)|_U \ge 0 \} 
    \]
    for any open $U \subset X$; we assume that $D + \div f \ge 0$ is satisfied for any $D$ if $f = 0$. We will occasionally omit the lower index $X$ and just write $\O(D)$.
\end{definition}

\begin{remark}
    For trivial reasons the subpresheaf $\O(D) \subset \underline{R(X)}$ is in fact an $\O_X$-submodule.
\end{remark}

\begin{remark}\label{prelim:div_lin_sys:rem:inequality_is_sheaf_inclusion}
    Let $D, E \in \Div X$ be divisors on a normal integral scheme $X$. Then $D \le E$ if and only if $\O(E) \subset \O(D)$ (as subsheaves of $\underline{R(X)}$).
\end{remark}
\begin{proof}
    If $D \le E$, then the inclusion $\O(E) \subset \O(D)$ is obvious. The converse: $D \le E$ can be checked at every local ring of dimension 1, so we could assume without loss of generality that $X = \Spec R$, where $R$ is a discrete valuation ring. Let $\xi$ be the unique closed point of $X$ and $\pi$ be a uniformizer, i.e. $\mathfrak m_\xi = (\pi)$. Then $D = d \cdot \xi$, $E = e \cdot \xi$ and $\O(D) = \pi^{-d} \cdot \O$, $\O(E) = \pi^{-e} \cdot \O$, so 
    \[
    D \le E \iff d \le e \iff -e \le -d \iff \O(E) \subset \O(D)
    \]
\end{proof}

\begin{claim}\label{prelim:div_lin_sys:claim:cartier_inv}
    Let $D$ be a Weil divisor on a normal integral noetherian scheme $X$. Then $D$ is Cartier if and only if $\O(D)$ is invertible.
\end{claim}
\begin{proof}
    $D$ is Cartier if and only if it is locally principle. $\O(D)$ is in invertible if and only if it is locally free. For any open $U \subset X$ the divisor $D$ is principal if and only if $\O(D|_U) \cong \O(D)|_U$ is trivial by the Hartogs' lemma. Thus we are done
\end{proof}

\begin{claim}\label{prelim:div_lin_sys:claim:divisorial_eq}
    If $D, E$ are Cartier divisors on a normal integral scheme $X$, then $D - E$ is principal if and only if $\O(D) \simeq \O(E)$. Now consider $\O(D)$, $\O(E)$ as fractional ideal sheaves, i.e. subsheaves of the sheaf of rational functions $\underline{R(X)}$. Then $D = E$ if and only if $\O(D) = \O(E)$.
\end{claim}
\begin{proof}
    \cite[Prop II.6.13]{Hartshorne1977}
\end{proof}

\begin{claim}\label{prelim:div_lin_sys:claim:product_sum}
    Let $D, E \in \CaDiv X$ be Cartier divisors on a normal integral noetherian scheme. Then $\O(D + E) \cong \O(D) \otimes \O(E)$. Furthermore, if we consider the sheaves $\O(D), \O(E), \O(D + E)$ as subsheaves of $\underline{R(X)}$, then $\O(D) \cdot \O(E) = \O(D + E)$
\end{claim}
\begin{proof}
    The first part is \cite[Prop II.6.13]{Hartshorne1977}. The second "Furthermore" part can be checked locally, so we can assume that $D = \div f$, $E = \div g$, and $D + E = \div fg$. Then $\O(D) = f^{-1} \cdot \O_X$, $\O(E) = g^{-1} \cdot \O_X$, and $\O(D + E) = (fg)^{-1} \O_X$. Hence we get:
    \[
    \O(D) \cdot \O(E) = f^{-1} \cdot g^{-1} \cdot \O_X = (fg)^{-1} \cdot \O_X = \O(D + E).
    \]
\end{proof}

\begin{corollary}\label{prelim:div_lin_sys:cor:dual_is_negiative}
    If $D$ is a Cartier divisor on a normal integral noetherien scheme, then $\O(-D) \cong \O(D)^\vee$   
\end{corollary}
\begin{proof}
    $\O(D) \otimes \O(-D) \cong \O$ by the above claim, so $\O(-D)$ is inverse to $\O(D)$, hence $\O(-D) \cong \O(D)^\vee$.
\end{proof}

Now we briefly discuss the notion of the support of a divisor.
\begin{definition}
    Let $X$ be a normal noetherian integral scheme, $D = \sum_Y n_Y \cdot Y$, where $Y$ runs over all prime divisors. Then we define the \textbf{support of the divisor $D$} as the subset $\Supp D \defeq \bigcup_{n_Y \ne 0} Y$.
\end{definition}
\begin{remark}
    Since $n_Y \ne 0$ only for finitely many $Y$, the subset $\Supp D \subset X$ is closed. 
\end{remark}
\begin{remark}
    $\Supp (- D) = \Supp D$.
\end{remark}

\subsubsection{Linear Systems} 
Until the end of this subsection $R$ is a (commutative) ring

\begin{definition}
    Let $X$ be a normal noetherian integral $R$-scheme. Then a \textbf{linear system} is a pair $\mathfrak d = (V, \mathcal L)$, where $\mathcal L$ is an invertible sheaf and $V \subset \Gamma(X, \mathcal L)$ is a non-zero $R$-submodule of the global sections of $\mathcal L$.
\end{definition}

\begin{definition}
    Let $(V, \mathcal L)$, $(V', \mathcal L')$ be linear systems . A \textbf{morphism of linear systems} is a morphism of sheaves $\mathcal L \to \mathcal L'$ such that the induced map $\Gamma(X, \mathcal L) \to \Gamma(X, \mathcal L')$ maps $V$ into $V'$. A morphism of linear systems is called an \textbf{isomorphism} if it admits an inverse morphism of linear systems.
\end{definition}

\begin{claim}\label{prelim:div_lin_sys:claim:cosection}
    Let $\mathcal L$ be an invertible sheaf on a non-empty normal integral scheme $X$ and $s \in \Gamma(X, \mathcal L)$ be a non-zero section. Consider the dual cosection $s^\vee: \mathcal L^\vee \to \O$ --- it is a monomorphis and there is a unique effective Cartier divisor $D$ such that $\O(-D) = s^\vee(\mathcal L^\vee)$ (they are equal as subsheaves of $\O$).
\end{claim}
\begin{proof}
    $\mathcal L^{\vee\vee} \cong \mathcal L$ and $s^{\vee\vee} = s$, so $s^\vee$ is non-zero. $\mathcal L^\vee$ is torsion-free, so we can check that $s^\vee$ is a mono only at the generic point of $X$ --- it would give a non-zero morphism $R(X) \to R(X)$, i.e. an isomorphism, in particular, a mono. Then $s^\vee(\mathcal L^\vee) \cong \mathcal L^\vee$, in particular it is invertible. By \cite[Prop II.6.13]{Hartshorne1977} there must be an effective Cartier divisor $D$ such that $\O(-D) = s^\vee(\mathcal L^\vee)$. By \autoref{prelim:div_lin_sys:claim:divisorial_eq} such divisor is unique.
\end{proof}

\begin{definition}
    Let $\mathfrak d  = (V, \mathcal L)$ be a linear system. Then for $s \in V \backslash 0$ we denote by $\div_\mathcal L s$ the effective Cartier divisor with the ideal sheaf $s^\vee(\mathcal L^\vee)$, i.e. such that $\O(-\div_\mathcal L s) = s^\vee(\mathcal L^\vee)$. For an effective Cartier divisor $D$ we write that $D \in \mathfrak d$ if there is $s \in V$ such that $D = \div_\mathcal L s$.
\end{definition}

\begin{example}
    Let $\mathcal L = \O(D)$ where $D$ is a Cartier divisor. Consider any function $f \in \Gamma(X, \O(D))$. Then $\div_\mathcal L f = D + \div f$.
\end{example}
\begin{proof}
    Local computation checks that the cosection $\O(-D) \cong \O(D)^\vee \to \O$ given by $f$ maps $g \mapsto g \cdot f$, so its image is $f \cdot \O(-D) = \O(-D - \div f)$ --- the ideal sheaf of the effective Cartier divisor $D + \div f$.
\end{proof}

\begin{claim}\label{prelim:div_lin_sys:claim:sections}
    Let $\mathcal L$ be an invertible sheaf on a non-empty normal noetherian integral scheme $X$ and $s, s' \in \Gamma(X, \mathcal L) \backslash 0$ be non-zero sections. Then $\div_\mathcal L s$ is linearly equivalent to $\div_\mathcal L s'$. Moreover, $\div_\mathcal L s = \div_\mathcal L s'$ if and only if there is an invertible function $f \in \Gamma(X, \O^\times)$ such that $s = f \cdot s'$.
\end{claim}
\begin{proof}
    We have $\O(-\div_\mathcal L s) = s^\vee(\mathcal L^\vee)$ and as we showed in the above proof the morphism $s^\vee: \mathcal L^\vee \to \O$ is a mono, in particular it is an isomorphism onto its image. Hence, $\O(-\div_\mathcal L s) \simeq \mathcal L^\vee$. The same way $\O(-\div_\mathcal L s') \simeq \mathcal L^\vee$. Therefore, $\O(-\div_\mathcal L s) \simeq \O(-\div_\mathcal L s')$, so $-\div_\mathcal L s$ is linearly equivalent to $-\div_\mathcal L s'$ by \autoref{prelim:div_lin_sys:claim:divisorial_eq} and $\div_\mathcal L s$ is linearly equivalent to $\div_\mathcal L s'$. 
    
    Now, if $\div s_\mathcal L = \div_\mathcal L s'$, then their ideal sheaves must coincide, i.e. $s^\vee(\mathcal L^\vee) = s'^\vee(\mathcal L^\vee)$. Hence, we must get an automorphism $(s^\vee)^{-1} \circ (s'^\vee): \mathcal L^\vee \to \mathcal L^\vee$. Since $\mathcal L^\vee$ is of rank 1 and is torsion-free, we have that all endomorphisms of $\mathcal L^\vee$ must come from multiplication by a rational function and all automorphisms come from multiplication by an invertible function\footnote{assume $f \in R(X)$ is such that both $f \cdot \mathcal L^\vee \subset \mathcal L^\vee$ and $f^{-1} \cdot \mathcal L^\vee \subset \mathcal L^\vee$ (as subsheaves of $\mathcal L^\vee \otimes R(X)$). Using regularity of $X$ in codimension 1 we get that $f$ is an invertible regular function in each local ring of dimension 1. Hence, by Hartogs' lemma $f \in \Gamma(X, \O^\times)$}.
\end{proof}

\begin{claim}\label{prelim:div_lin_sys:claim:sec_vanishing_criterion}
    Let $f: X' \to X$ be a morphism of non-empty normal noetherian integral schemes, $\mathcal L$ be a invertible sheaf, $s \in \Gamma(X, \mathcal L)$ be a global section. Then $f^* s = 0 \iff f(X') \subset \div_\mathcal L s$.
\end{claim}
\begin{proof}
    Since $s$ is a morphism of locally free sheaves, we have that $f^* s^\vee = (f^* s)^\vee$ and that $f^*s = (f^* s^\vee)^\vee$, in particular $f^*s = 0 \iff f^* s^\vee = 0$. Now, $f^* s^\vee = 0$ if and only if $s^\vee(\mathcal L^\vee) \subset \Ker (\O_X \to f_* \O_{X'}) = \mathcal I_{f(X')}$ --- which means exactly that $f(X') \subset \div_{\mathcal L} s$.
\end{proof}

\subsection{Equivariance}\label{prelim:ssec:equivariance}
Let $G$ be a group variety over the field $k$, $X$ be a variety over $k$ and $G \times_k X \to X$ be a $k$-action. We say that an invariant of $U$ of $X$ is $G$-equivariant if the base change $U_{\bar k}$ is preserved under the action of $G(\bar k)$ on $X_{\bar k}$. We also give a more formal definition and then a few examples that we will use. One does not need to understand the formal definition to fully comprehend the examples.

By an invariant of a variety we mean the following. Let $V_k$ be a category such that $Ob(V_k)$ are morphisms $Y \to \Spec L$ , where $L$ is any field extension of $k$ and $Y$ is an $L$-scheme and arrows in $V_k$ are commutative diagrams\footnote{i.e. $V_k$ is just the comma category corresponding to $Sch \to Sch \xleftarrow{\Spec} (k/Fields)^\circ$}. Let $\mathcal C \subset V_k$ be a subcategory closed under field extensions and $T: \mathcal C^\circ \to Sets$ be a functor ($\mathcal C^\circ$ is the opposite category). An invariant of $X$ is an element $U \in T(X)$. Denote by $U_L$ the image of $U$ under $T(X) \to T(X_L)$. We say that $U$ is $G$-equivariant if for any field extension $L/k$ and for any $g \in G(L)$ we have that $T(g) (U_L) = U_L$.
\begin{example}
    Here we define the notion of a $G$-equivariant divisor on a variety. We have the natural homomorphism $\Div X \to \Div X_L$, $D \mapsto D_L$ for any field extension $L/k$. So, we say that a divisor $D \in \Div X$ is $G$-equivariant if for any $g \in G(\bar k)$ we have that $g^* D_{\bar k} = D_{\bar k}$. Clearly $G$-equivariant divisors form a subgroup of $\Div X$ -- we denote it by $\Div_G X$. Clearly the homomorphism $\Div X \to \Div X_L$ maps Cartier divisors to Cartier divisors, so we get $\CaDiv X \to \CaDiv X_L$ and define $\CaDiv_G X_L$ the same way.

    Making it formal, here $\mathcal C$ is the subcategory of $V_k$ of varieties that are regular in codimension 1 and $T = \Div$. To define $T(f): \Div Y \to \Div Y'$ for $f: Y' \to Y$ it is sufficient to define $T(f)$ on prime divisors, which is clear. Note, that since we are working with varieties, instead of checking equivariance for all field extensions, it is sufficient to check only the equivariance with respect to $\bar k$
\end{example}

\begin{example}
    Now we define $G$-equivariant subsheaf of $\O_X^{\oplus m}$. Coherent subsheaf $\mathcal F \subset \O_X^{\oplus m}$ is called $G$-equivariant if for any $g \in G(\bar k)$ the image of $\mathcal F_{\bar k} \subset \O_{X_{\bar k}}^{\oplus m}$ under $\O_{X_{\bar k}}^{\oplus m} \to g_* \O_{X_{\bar k}}^{\oplus m}$ coincides with $g_* \mathcal F_{\bar k} \subset g_* \O_{X_{\bar k}}^{\oplus m}$.

    Again, using our formalism, we can take $\mathcal C = V_k$ and $T = \{$quasicoherent subsheaves of $\O_X^{\oplus m}\}$. Note that instead of taking $\O_X^{\oplus m}$ we can take $\O_X \otimes V$, where $V$ is a fixed vector space over $k$.
\end{example}

\begin{example}
    We conclude with $G$-equivariant linear systems. Let $(V, \mathcal L)$ be a linear system on $X$. Then we naturally have $\O \otimes V \to \mathcal L$. By \autoref{prelim:div_lin_sys:claim:cosection} the dual morphism $\mathcal L^\vee \to \O \otimes V^\vee$ is a mono. We say that the linear system $(V, \mathcal L)$ is $G$-equivariant if $\Im(\mathcal L^\vee \to \O \otimes V^\vee)$ is $G$-equivariant as a subsheaf of a free sheaf.
\end{example}

\subsection{Newton Polytope Theory}\label{prelim:ssec:newton_polytopes}
In this section we give a very short overview of the Newton Polytope Theory. Essentially Newton Polytope Theory is the study of complete intersections in torus (or, more generally, in toric varieties) in terms of monomials of their equations. 

Recall that $M$ is the character lattice of $T^n$. For a Laurent polynomial $f =~\sum_{\chi \in M} c_\chi \cdot~\chi$ from $k[T^n]$ we define the support set of $f$ as follows:
\[
\Supp f \defeq \Big\{ \chi \in M\ |\ c_\chi \ne 0 \Big\}.
\]
Then we can define $k^A \defeq \{ f \in k[T^n]\ |\ \Supp f \subset A \}$ --- the space of polynomials supported at $A$.

\subsubsection{Kouchnirenko-Bernstein Theorem}
Here we recall a classical result that laid the foundations of Newton Polytope theory.

\begin{definition}
    For two subsets $A, B \subset \R^n$ we define $A + B \defeq \{ a + b\ |\ a \in A, b \in B\}$ --- \textbf{the Minkowski sum}.
\end{definition}

\begin{definition}
    Let $L$ be a lattice, i.e. $L \simeq \Z^n$. We define the \textbf{lattice volume} with respect to $L$ as the unique Euclidean volume form $\Vol_L$ on $L_\R$ such that $\Vol_L (\Delta) = 1$, where\footnote{by $\Conv$ we denote the convex hull} $\Delta = \Conv \{ 0, e_1, \dots, e_n \}$ and $e_1, \dots, e_n$ is a basis of $L$.
\end{definition}

\begin{remark}
    For any finite subset $S \subset L$ we have that $\Vol_L(\Conv S)$ is an integer because $\Conv S$ admits a triangulation by simplicies with vertices in $L$.
\end{remark}

\begin{remark}
    Recall that a polytope is the convex hull of finitely many points. One can easily see that $\Conv (A + B) = \Conv A + \Conv B$, so the sum of any two polytopes is a polytope. It means that given a real space $V$ the set of all polytopes $\operatorname{Pol}(V)$ from $V$ is naturally a monoid with the operation of Minkowski sum and $\{ 0 \}$ as the neutral element.
\end{remark}

\begin{definition}
    Let $L$ be a lattice of rank $n$. The \textbf{lattice mixed volume} with respect to $L$ is the unique function $\MVol_L: \operatorname{Pol}(L_\R)^n \to \R_+$ that satisfies:
    \begin{itemize}
        \item \textit{Linearity:} $\MVol_L(P_1 + P', P_2, \dots P_n) = \MVol_L (P_1, P_2, \dots, P_n) + \MVol_L(P', \dots, P_n)$ for all $P', P_i \in \operatorname{Pol}(L_\R)$;
        
        \item \textit{Symmetricity:} $\MVol_L(P_1, \dots, P_n) = \MVol_L(P_{\sigma(1)}, \dots, P_{\sigma(n)})$ for all $\sigma \in S_n$ and $P_i \in\operatorname{Pol}(L_\R)$;

        \item \textit{Diagonal volume:} $\MVol_L(P, \dots, P) = \Vol_L (P) \quad \forall P \in \operatorname{Pol}(V)$.
    \end{itemize}
    In other word, $\MVol_L$ is the polarization of $\Vol_L: \operatorname{Pol}(L_\R) \to \R_+$.
\end{definition}

\begin{claim}
    $\MVol_L(P_1, \dots, P_n) = \frac{1}{n!} \sum_{l = 1}^n (-1)^{n - l} \sum_{1 \le i_1 < \dots < i_l \le n} \Vol_L (P_{i_1} + \dots + P_{i_l})$.
\end{claim}
\begin{proof}
    Cf. \cite[Thm 3.7, p.118]{Ewald1996}.
\end{proof}

\begin{remark}
    For any subsets $S_1, \dots, S_n \subset L$ we have that $\MVol(\Conv S_1, \dots, \Conv S_n)$ is an integer.
\end{remark}

\begin{theorem}[Kouchnirenko-Bernstein]
    Let $A_1, \dots, A_n \subset M$ be finite subsets of the character lattice and $\Delta_i \defeq \Conv_{M_\R} A_i$ be the corresponding Newton Polytopes. Let $k = \bar k$. Then for the general $\mathbf f \in k^{A_\bullet} \defeq k^{A_1} \times \dots \times k^{A_n}$ the square system $f_1 = \dots = f_n = 0$ has $\MVol_M(\Delta_1, \dots, \Delta_n)$ solutions in $T^n$.
\end{theorem}
\begin{proof}
    See \cite{bernstein} for the case $k = \C$ and for the arbitrary field see \cite{Kushnirenko1977} --- the author wrote the proof only for $k = \C$ but since it is purely algebraic the proof is valid over arbitrary algebraically closed field. In fact the proof in \cite{bernstein} also does not rely on any techinques that work exclusively in zero characteristic so it may be adapted to work in a purely algebraic setting as well.
\end{proof}

\subsection{Khovanskii Theorems and Irreducible Components}
\label{prelim:ssec:khovanskii}

\begin{definition}
    Let $X$ be a variety over $k$. Then the number of geometrically irreducible components of $X$ is the number of irreducible components of $X_{\bar k}$
\end{definition}

\begin{definition}\label{prelim:khovanskii:def:K_number}
    Let $A_1, \dots, A_m \subset M$ be finite subsets. We denote by $K_{T^n}(A_1, \dots, A_m)$ the number of geometric irreducible components of the variety cut out in $T^n$ by the equations $f_1 = \hdots = f_m = 0$ for the general $f_1, \dots, f_m$ such that $\Supp f_i \subset A_i$.
\end{definition}

\begin{definition}\label{prelim:khovanskii:def:defect}
    For a collection $A_1, \dots, A_m \subset M$ of finite subsets and a non-empty indices subset $J \subset ~\{ 1, \dots, m \}$ we define \textbf{the defect of $J$}: $\delta(J) \defeq \dim \left( \sum_{j \in J} A_j \right) -~|J|$.
\end{definition}

\begin{theorem}
    Let $A_1, \dots, A_m \subset M$ be finite subsets. Denote by $\Delta_i$ the convex hulls of $A_i$ in $M_\R$. Then:
    \begin{enumerate}
        \item If $\delta(J) > 0$ for all non-empty $J \subset \{1, \dots, m \}$, then $K_{T^n}(A_1, \dots, A_m) = 1$.

        \item If there is $J \subset \{1, \dots, m \}$ such that $\delta(J) < 0$, then $K_{T^n}(A_1, \dots, A_m) = 0$.

        \item If $\delta(J) \ge 0$ for all non-empty $J \subset \{1, \dots, m \}$ and for some non-empty subset the defect is zero, then there is the greatest subset $J_0$ such that $\delta(J_0) = 0$ and $K_{T^n}(A_1, \dots, A_m) =~\MVol_L (\Delta_j)_{j \in J_0}$, where $L$ is the minimal saturated sublattice of $M$ such that\footnote{i.e. $\chi_i \chi_j^{-1} \in L \ \forall \chi_i \in A_i$, $\chi_j \in A_j$ for any two $i, j \in J_0$.} $A_i - A_j \subset L$ for all $i, j \in J_0$
    \end{enumerate}
\end{theorem}
\begin{proof}
    For the case $k = \C$ see \cite[Th. 17, Th. 19]{Khovanskii2016}, for the arbitrary field case see \cite[Th. 4.4]{Zhizhin2024}
\end{proof}
\section{Equivariant Linear Systems}\label{sec:linsys}
The ultimate goal of this section is to define the combinatorial datum of an equivariant linear system on a toric variety and then to develop the technique for restricting that datum to oribts of the toric variety. In the first two subsections we discuss the general notion of divisor pullback (\autoref{linsys:ssec:pullbacks_div}) and linear system pullback (\autoref{linsys:ssec:pullbacks_ls}) --- the material is nothing new and rather well-known or at least intuitive, but we did not find a formal and complete enough treatise of the subject in the literature. Then in \autoref{linsys:ssec:equivar_div} we discuss the equivariant toric picture, in particular we show that the divisor suuport function behaves well with respect to pullbacks. Finally, in \autoref{linsys:ssec:equivar_ls} we define the combinatiral datum of an equivariant linear system on a toric variety and show how to compute the datum of the restriction of a linear system to an orbit.

\subsection{Pullbacks of Divisors}\label{linsys:ssec:pullbacks_div}
\begin{definition}
    Let $f: X' \to X$ be a morphism of non-empty normal integral noetherian schemes and $\eta$ be the image of the generic point of $X'$. We say that a Cartier divisor $D \in \CaDiv X$ \textbf{lifts along} $f$ if $\O_X(D)_\eta = \O_{X, \eta}$, where $\O_X(D)_\eta$, $\O_{X, \eta}$ are considered as $\O_{X, \eta}$-submodules of $R(X)$. Otherwise we say that $D$ degenerates along $f$.
\end{definition}

\begin{remark}
    The idea of the above definition is rather simple. The condition that $\O_X(D)_\eta \subset \O_{X, \eta}$ means that $\eta$ (and hence $f(X)$) is not contained in a "pole" of $D$ (i.e. a prime component of $D$ with negative coefficient). Given that $\O_X(D)_\eta \subset \O_{X, \eta}$ the condition $\O_X(D)_\eta = \O_{X, \eta}$ is equivalent to the condition that $\O_X(D)_\eta$ is not contained in the maximal ideal $\O_{X, \eta}$, i.e. $\eta$ (and hence $f(X)$) is not contained in a "zero" of $D$. So we are effectively requiring that $f(X) \not \subset \Supp D$ (as proved in the theorem bellow).
\end{remark}

\begin{remark}\label{linsys:pullbacks_div:rem:principal}
    Let $f: X' \to X$ be a morphism of non-empty normal notherian integral schemes with $\eta$ being the image of the generic point of $X'$. Let $g \in R(X)^\times$. Then $\div g$ does not degenerate along $f$ if and only if $g \in \O_{X, \eta}^\times$.
\end{remark}

\begin{claim}
    If $f: X' \to X$ is a dominant morphism of non-empty normal noetherian integral schemes, then any Cartier divisor lifts long $f$.
\end{claim}
\begin{proof}
    Let $\eta$ be the image under $f$ of the generic point of $X'$. As $f$ is dominant, $\eta$ must be the generic point of $X$. Then for any Cartier divisor $D \in \CaDiv X$, as $\O(D)$ is of rank 1, we have that $\O(D)_\eta = R(X) = \O_{X, \eta}$.
\end{proof}

\begin{claim}\label{linsys:pullbacks_div:claim:sandwich}
    Let $f: X' \to X$ be a morphisms of non-empty normal integral noetherian schemes and $D, D', D'' \in \CaDiv X$ be Cartier divisor such that $D' \le D \le D''$ and $D'$, $D''$ lift along $f$. Then $D$ lifts along $f$
\end{claim}
\begin{proof}
    By \autoref{prelim:div_lin_sys:rem:inequality_is_sheaf_inclusion} $\O(D'') \subset \O(D) \subset \O(D')$. Denote by $\eta$ the image of the generic point of $X'$. Then $\O_{X, \eta} = \O(D'')_\eta \subset \O(D) \subset \O(D')_\eta = \O_{X, \eta}$, hence $\O(D)_\eta =~\O_{X, \eta}$, so $D$ lifts along $f$ by definition.
\end{proof}

\begin{corollary}\label{linsys:pullbacks_div:claim:sandwich_eff}
    Let $f: X' \to X$ be a morphisms of non-empty normal integral noetherian schemes and $D, E \in \CaDiv^+ X$ be effective Cartier divisor such that $E \ge D$ and $E$ lifts along $f$. Then $D$ lifts along $f$
\end{corollary}
\begin{proof}
    The zero divisor $0$ lifts along any morphism. So, we have that $0 \le D \le E$ and $0, E$ lift along $f$. Hence, $D$ lifts along $f$.
\end{proof}

\begin{theorem}\label{linsys:pullbacks_div:thm:equiv}
    Let $f: X' \to X$ be a morphism of non-empty normal noetherian integral schemes and $D \in \CaDiv X$. Then the following are equivalent:
    \begin{enumerate}[1)]
        \item $D$ lifts along $f$;
        \item $-D$ lifts along $f$;
        \item there is $p \in X'$ such that $\O_X(D)_{f(p)} = \O_{X, f(p)}$;
        \item $f(X') \not \subset \Supp D$, where $\displaystyle\Supp \sum_{Y \text{ -- prime}} n_Y \cdot Y \defeq \bigcup_{n_Y \ne 0} Y$;
        \item let $\eta \in X$ be the image of the generic point of $X'$, then $\O_X(D) \to \underline{R(X)}$ factors through $\underline{\O_{X, \eta}} \to \underline{R(X)}$ and the pulled back morphism $f^* \O_X(D)\to~f^*\underline{\O_{X, \eta}} =~\underline{R(X')}$ is non-zero.
    \end{enumerate}
\end{theorem}
\begin{proof}
    \begin{description}
        \item[$1) \iff 2)$] Follows immediately from \autoref{prelim:div_lin_sys:cor:dual_is_negiative}.
        
        \item[$1) \implies 3)$] By definition we could take the generic point of $X'$ as $p$.
        
        \item[$3) \implies 4)$] Let $D = \sum_{Y \text{ -- prime}} n_Y \cdot Y$ and $Y$ be such that $p \in Y$ and let $\xi \in X$ be the generic point of $Y$. Then $\O_X(D)_\xi = \left(\O_X(D)_{f(p)}\right)_\xi = \O_{X, \xi}$ as $\O_{X, \xi}$-submodules of $R(X)$. In particular, $n_Y = 0$. Hence, $p \not \in \Supp D$.
        
        \item[$4) \implies 1)$] Let $\eta' \in X'$ be the generic point and $\eta \defeq f(\eta')$. Again, let $D =~\sum_{Y \text{ -- prime}} n_Y \cdot~Y$. Since $n_Y \ne 0$ only for finitely many $Y$, we have that $\Supp D$ is closed. As $\overline{\eta} = \overline{f(X)}$, if $\eta \in \Supp D$, then $f(X') \subset \Supp D$. Hence, $\eta \not \in \Supp D$. Denote $U \defeq X \backslash \Supp D$. Obviously $\O_X(D)|_U = \O_U$ (as $\O_U$-submodules of $\underline{R(X)}$). Since $\eta \in U$, we are done. 
        
        \item[$5) \iff 1)$]
        First note that for any quasi-coherent $\O_X$-module $\mathcal F$ and for any point $p\in~X$ we have the natural isomorphisms $\Hom_{\O_X}\left(\mathcal F, \underline{\O_{X, p}}\right) \cong \Hom_{\O_{X, p}}\left(\mathcal F_p, \O_{X, p}\right)$ and $\Hom_{\O_X}\left(\mathcal F, \underline{R(X)}\right) \cong~\Hom_{\O_{X, p}}\left(\mathcal F_p, R(X)\right)$. Putting $p = \eta$ and $\mathcal F = \O(D)$ we get that if $D$ does not degenerate along $f$, then we have the canonical factorization $\O_X(D) \to \underline{\O_X(D)_\eta} = \underline{\O_{X, \eta}} \to \underline{R(X)}$.
        
        Now we assume that we have the factorization $\O_X(D) \to \underline{\O_{X, \eta}} \to \underline{R(X)}$ and prove the equivalence. Like before, for any quasi-coherent $\O_{X'}$-module $\mathcal F$ we have that
        \[
        \operatorname{Hom}_{\O_{X'}}\left(\mathcal F, \underline{R(X')}\right) \cong \operatorname{Hom}_{R(X')}(\mathcal F \otimes R(X'), R(X')),
        \]
        where by $\mathcal F \otimes R(X)$ we mean the fiber\footnote{which is the same as the stalk in this case} of $\mathcal F$ at the generic point of $X'$. For any morphism of $\O_X$-modules $\mathcal G_1 \to \mathcal G_2$ we have that\footnote{the isomorphism takes place in the category of arrows of vector spaces over $R(X)$} 
        \[
        (f^* \mathcal G_1 \to f^*\mathcal G_2) \otimes R(X') \cong(\mathcal G_1 \otimes k(\eta) \to \mathcal G_2 \otimes k(\eta)) \otimes_{k(\eta)} R(X),
        \]
        where $k(\eta)$ is the residue field of the local ring $\O_{X, \eta}$. Now, $D$ lifts along $f$ $\iff$ $\O_X(D)_\eta \to \O_{X, \eta}$ is an iso $\iff$ $\O_X(D)_\eta \to \O_{X, \eta}$ is surjective $\iff$ $\O_X(D) \otimes k(\eta) \to k(\eta)$ is surjective (Nakayama's lemma) $\iff$ the morphism $(\O_X(D) \otimes k(\eta)) \otimes R(X') \to R(X')$ is surjective (field extension does not affect surjectivity) $\iff$ $(\O_X(D) \otimes k(\eta)) \otimes R(X') \to R(X')$ is non-zero ($R(X')$ is 1-dimensional over itself).
    \end{description}
\end{proof}

\begin{corollary}\label{linsys:pullbacks_div:cor:pullback_of_a_divisor_is_divisor}
    Let $f: X' \to X$ be a morphism of non-empty normal integral noetherian schemes and $D \in \CaDiv X$ be a Cartied divisor that lifts along $f$. Then by the virtue of item 5) of \autoref{linsys:pullbacks_div:thm:equiv} we get the morphism $f^* \O(D) \to \underline{R(X)}$. It is a monomorphism and its image is an invertible subsheaf of $\underline{R(X')}$.
\end{corollary}
\begin{proof}
    $f^* \O_X(D)$ is invertible (as a pullback of an invertible sheaf), in particular it is torsion-free, so we can check that $f^* \O_X(D) \to \underline{R(X')}$ is mono only at the generic point where it is clear, because $\rk f^* \O_X(D) = 1$ and the morhpism $\O_X(D) \to \underline{R(X')}$ is non-zero (by item 5 of \autoref{linsys:pullbacks_div:thm:equiv}) and every non-zero morphism $R(X) \to R(X)$ is a mono. Since $f^* \O_X(D) \to \underline{R(X)'}$ is a mono, its image is isomoprhic to $f^* \O_X(D)$, so it is invertible. 
\end{proof}

\begin{definition}
    Let $f: X' \to X$ be morphism of non-empty normal noetherian integral schemes and $D \in \CaDiv X$ be a Cartier divisor that lifts along $f$. Then by the virtue of item 5) of \autoref{linsys:pullbacks_div:thm:equiv} and \autoref{linsys:pullbacks_div:cor:pullback_of_a_divisor_is_divisor} we get a mono $f^* \O_X(D) \to ~\underline{R(X')}$ --- and its image is an invertible sheaf of fractional ideal, so by \cite[Prop. II.6.13]{Hartshorne1977} we get that there must be a unique Cartier divisor on $X'$ that we denote by $f^*D \in \CaDiv X'$ such that $\Im\Big(f^* \O_X(D) \to \underline{R(X')}\Big) = \O_{X'}(f^* D)$. We call $f^*D$ \textbf{the pullback of $D$ along $f$}.
\end{definition}

\begin{remark}\label{linsys:pullbacks_div:rem:pullback_commutes_w_O}
    Let $f: X' \to X$ be a morphism of non-empty normal noetherian integral schemes and $D \in \CaDiv X$ be a Cartier divisor that lifts along $f$. Then $f^*\O(D) \cong \O(f^* D)$.
\end{remark}
\begin{proof}
    By definition $\O(f^*D)$ is the image of $f^*\O(D) \to \underline{R(X')}$ and by \autoref{linsys:pullbacks_div:cor:pullback_of_a_divisor_is_divisor} $f^* \O(D) \to \underline{R(X')}$ is a mono, so we have the natural iso $f^*\O(D) \to \O(f^* D)$.
\end{proof}

\begin{definition}
    Let $f: X' \to X$ be morphism of non-empty normal integral noetherian schemes. We denote by $\CaDiv_f X$ the set of divisors that lift along $f$.
\end{definition}

\begin{claim}
    $\CaDiv_f X$ is a subgroup of $\CaDiv X$ and $f^*: \CaDiv_f X \to \CaDiv X'$, $D \mapsto f^*D$ is a group homomorphism.
\end{claim}
\begin{proof}
    From \autoref{linsys:pullbacks_div:thm:equiv} we already know that $\CaDiv_f X$ is closed under taking additive inverse. So it remains to show that for $D, E \in \CaDiv_f X$ we have that $D + E \in \CaDiv_f X$ and $f^*(D + E) = f^*D + f^* E$. Let $\eta$ be the image of the generic point of $X'$. From \autoref{prelim:div_lin_sys:claim:product_sum} we know that $\O_X(D + E) = \O_X(D) \cdot \O_X(E)$, so 
    \[
    \O_X(D + E)_\eta = \O_X(D)_\eta \cdot \O_X(E)_\eta = \O_{X, \eta} \cdot \O_{X, \eta} = \O_{X, \eta}
    \]
    which means that $D + E$ lifts along $f$. Finally: 
    \begin{multline*}
        \O_{X'}(f^*(D + E)) = \Im \Big( f^*\O_X(D + E) \to \underline{R(X')}\Big) = \Im \Big( f^*\big(\O_X(D) \cdot \O_X(E)\big) \to \underline{R(X')}\Big) = \\
        = \Im \Big( f^*\O_X(D) \to \underline{R(X')}\Big) \cdot \Im \Big( f^*\O_X(E) \to \underline{R(X')}\Big) = \O_{X'}(f^*D) \cdot \O_{X'}(f^* E) = \O_{X'}(f^*D + f^*E),
    \end{multline*}
    so $f^*(D + E) = f^*D + f^* E$.
\end{proof}

\begin{claim}\label{linsys:pullbacks_div:claim:supp_lifting}
    Let $f: X' \to X$ be a morphism non-empty normal noetherian schemes, $D$ be a Cartier divisor on $X$ that lifts along $f$. Then $\Supp f^*D \subset f^{-1} (\Supp D)$. Moreover, if $D$ if effective, then $\Supp f^* D = f^{-1} (\Supp D)$.
\end{claim}
\begin{proof}
    Fix any $p \in X'$. We want to show that if $p \in \Supp f^* D$, then $p \in f^{-1} (\Supp D)$, i.e. $f(p) \in \Supp D$ and if $D$ is effective, then the converse holds as well. We could assume without loss of generality that $X' = \Spec \O_{X', p}$ and $X = \Spec \O_{X, f(p)}$ --- in particular, now $X, X'$ are affine and $D = \div g$ for some $g \in R(X)^\times$, so $D = \div f^* g$. If $p \in \Supp f^* D$, then $f^* g \in R(X') \backslash \O_{X', p}^\times$. But $f^*(\O_{X, f(p)}^\times) \subset \O_{X', p}^\times$, so $g \not \in \O_{X, f(p)}^\times$, hence $f(p) \in \Supp D$. Now, assume that $D$ is effective --- then $g \in \O_{X, f(p)}$. If $f(p) \in \Supp D$, then $g \not \in \O_{X, f(p)}^\times$. As $\O_{X, f(p)}$ is a local ring, we get that $g \in \mathfrak m_{f(p)}$. Finally, $f^*\mathfrak m_{f(p)} \subset \mathfrak m_p$, so $f^* g \not \in \O_{X', p}^\times$ and $p \in \Supp f^* D$.
\end{proof}

\begin{claim}\label{linsys:pullbacks_div:claim:functorial_non_eff}
    Let $f: X' \to X$ be a morphisms of non-empty normal noetherian integral schemes and $D \in \CaDiv X$ be a Cartier divisor that lifts along $f \circ g$. Then $D$ lifts along $f$, $f^* D$ lifts along $g$ and $(f \circ g)^* D = g^*(f^*D)$
\end{claim}
\begin{proof}
    If $D$ lifts along $f \circ g$, then by \autoref{linsys:pullbacks_div:thm:equiv} $f(g(X'')) \not \subset \Supp D$. Since $f(g(X'')) \subset f(X')$, we get that $f(X') \not \subset \Supp D$, so by \autoref{linsys:pullbacks_div:thm:equiv} $D$ lifts along $f$. Again, $f(g(X'')) \not \subset \Supp D$, therefore $g(X') \not \subset f^{-1}(\Supp D)$. By  \autoref{linsys:pullbacks_div:claim:supp_lifting} $\Supp f^* D \subset f^{-1}(\Supp D)$, so $g(X'') \not \subset f^* D$ and by \autoref{linsys:pullbacks_div:thm:equiv} $f^* D$ lifts along $g$. The identity $(f \circ g)^* D = g^* (f^* D)$ can be checked locally on $X$ and for principal divisors it is trivial, so we are done. 
\end{proof}

\begin{claim}\label{linsys:pullbacks_div:claim:functorial_eff}
    Let $g: X'' \to X'$, $f: X' \to X$ be morphisms of non-empty normal noetherian integral schemes and $D \in \CaDiv^+ X$ be an effective Cartier divisor. Then $D$ lifts along $f \circ g$ if and only if $D$ lifts along $f$, $f^* D$ lifts along $g$. Moreover, in that case $(f \circ g)^* D = g^*(f^*D)$
\end{claim}
\begin{proof}
    Given \autoref{linsys:pullbacks_div:claim:functorial_non_eff} we only need to check that if $D$ lifts along $f$ and $f^*D$ lifts along $g$, then $D$ lifts along $f \circ g$. Since $f^*D$ lifts along $g$, by \autoref{linsys:pullbacks_div:thm:equiv} $g(X'') \not \subset~\Supp f^* D$. By \autoref{linsys:pullbacks_div:claim:supp_lifting} $\Supp f^*D = f^{-1}(\Supp D)$. So, $g(X'') \not \subset~f^{-1}(\Supp D)$ implies that $f(g(X'') \not \subset \Supp D$, thus by \autoref{linsys:pullbacks_div:thm:equiv} we are done. 
\end{proof}

\begin{claim}\label{linsys:pullbacks_div:claim:immersion_closure}
    Let $X$ be a non-empty normal noetherian integral scheme, $Z \xhookrightarrow{} X$ be an immersion\footnote{i.e. composition of an open embedding and a close embedding, by \citestacks{01QV} the order of embeddings is not important} (with $Z$ being integral) and $D \in \CaDiv X$ be a Cartier divisor. Then $D$ lifts along $Z \xhookrightarrow{} X$ if and only if $D$ lifts along $\bar Z \xhookrightarrow{} X$.
\end{claim}
\begin{proof}
    The image of the generic point of $Z$ is the same as the image of the generic point of $\bar Z$ 
\end{proof}

\subsection{Pullbacks of Linear Systems}\label{linsys:ssec:pullbacks_ls}

\begin{definition}
    Let $f: X' \to X$ be a morphism of non-empty normal integral noetherian schemes. We say that \textbf{the linear system $\mathfrak d = (V, \mathcal L)$ lifts along $f$} if the natural morphism $V \to \Gamma(X', f^* \mathcal L)$ is non-zero. Otherwise we say that \textbf{$\mathfrak d$ degenerates along $f$}.
\end{definition}

\begin{definition}
    Let $f: X' \to X$ be a morphism of non-empty normal noetherian integral schemes and $\mathfrak d = (V, \mathcal L)$ be a linear system on $X$ that lifts along $f$. Then we define the \textbf{pullback of} $\mathfrak d$ as the linear system $f^* \mathfrak d \defeq (V', f^*\mathcal L)$, where $V' \defeq \Im(V \to \Gamma (X', f^* \mathcal L))$.
\end{definition}

\begin{claim}\label{linsys:pullbacks_linsys:claim:criterion}
    Let $f: X' \to X$ be morphism of non-empty normal integral noetherian schemes and $\mathfrak d$ be a linear system on $X$. Then $\mathfrak d$ lifts along $f$ if and only if there is $D \in \mathfrak d$ such that $D$ lifts along $f$.
\end{claim}
\begin{proof}
    Assume that $\mathfrak d = (V, \mathcal L)$ lifts along $f$. Then there is $s \in V$ such that $f^*s$ is a non-zero section of $f^* \mathcal L$. Let $\eta' \in X'$ be the generic point and $\eta \defeq f(\eta')$. We know that $f^*s_{\eta'} \ne 0$. As $\O_{X', \eta'} = R(X')$, we have that $f^*s_{\eta'} = f^*s \otimes k(\eta') = f^*(s \otimes k(\eta))$, so $s \otimes k(\eta)$ is non-zero. $\mathcal L \otimes k(\eta) \simeq k(\eta)$ and $s \otimes k(\eta): k(\eta) \to \mathcal L \otimes k(\eta)$ is non-zero, so it is an isomorphism, hence $s^\vee \otimes k(\eta) = (s \otimes k(\eta))^\vee: \mathcal L^\vee \otimes k(\eta) \to k(\eta)$ is an iso. Now, we have that $\O_X(-\div_\mathcal L s) \subset \O_X$ and that the inclusion (which is just $s^\vee$) becomes an isomorphism when tensored by $k(\eta)$. By Nakayama's lemma $\O_X(-\div_\mathcal L s)_\eta = \O_{X, \eta}$. Hence, $-\div_\mathcal L s$ lifts along $f$, so by \autoref{linsys:pullbacks_div:thm:equiv} $\div_\mathcal L s$ lifts along $f$ (and $\div_\mathcal L s \in \mathfrak d$).

    Now we prove the converse. Assume that there is $D \in \mathfrak d$ that lifts along $f$. As before, $\mathfrak d = (V, \mathcal L)$. Let $s \in V$ be such that $D = \div_\mathcal L s$. By \autoref{linsys:pullbacks_div:thm:equiv} $-D$ lifts along $f$ as well. By \autoref{linsys:pullbacks_div:thm:equiv} there is $p \in X'$ such that $\O_X(-D)_{f(p)} = \O_{X, f(p)}$. The inclusion $\O_X(-D) \subset \O_X$ is $s^\vee: \mathcal L^\vee \to \O_X$, so we get that $s_{f(p)}: \O_{X, f(p)} \to \mathcal L_{f(p)}$ is an iso. Hence, $s \otimes k(f(p)): k(f(p)) \to \mathcal L \otimes k(f(p))$ is an iso. Base change of an isomorphism is an isomorphism, thus $f^*s \otimes k(p) = \Big(s \otimes k(f(p))\Big)\otimes_{k(f(p))} k(p)$ is an isomorphism. It means that $f^* s$ is non-zero at $p$, so $f^*s$ is a non-zero section of $f^* \mathcal L$, hence $\mathfrak d$ lifts along $f$.
\end{proof}

\begin{theorem}
    Let $f: X' \to X$ be a morphism of non-empty normal noetherian integral schemes and $\mathfrak d = (V, \mathcal L)$ be a linear system. Then the following are equivalent:
    \begin{enumerate}
        \item $\mathfrak d$ lifts along $f$;
        \item There is $p \in X'$ such that $V \to \mathcal L \otimes k(f(p))$ is non-zero;
        \item $f(X') \not \subset \bigcap_{D \in \mathfrak d} D$.
    \end{enumerate}
\end{theorem}
\begin{proof}
    Follows immediately from \autoref{linsys:pullbacks_div:thm:equiv} and \autoref{linsys:pullbacks_linsys:claim:criterion}.
\end{proof}

\begin{claim}\label{linsys:pullbacks_linsys:claim:immersion_closure}
    Let $g: X''\to X'$, $f: X' \to X$ be morphisms of non-empty normal integral noetherian schemes and $\mathfrak d$ be a linear system on $X$. Then $\mathfrak d$ lifts along $f \circ g$ if and only if $\mathfrak d$ lifts along $f$ and $f^*\mathfrak d$ lifts along $f$. Moreover, in that case $(f \circ g)^* \mathfrak d \cong g^*(f^* \mathfrak d)$
\end{claim}
\begin{proof}
    Follows immediately from \autoref{linsys:pullbacks_div:claim:functorial_eff} and \autoref{linsys:pullbacks_linsys:claim:criterion}.
\end{proof}

\begin{claim}\label{linsys:pullbacks_linsys:claim:immersion_closure}
    Let $X$ be a non-empty normal noetherian integral scheme, $Z \xhookrightarrow{} X$ be an immersion (with $Z$ being integral) and $\d$ be a linear system on $X$. Then $\d$ lifts along $Z \xhookrightarrow{} X$ if and only if $\d$ lifts along $\bar Z \xhookrightarrow{} X$.
\end{claim}
\begin{proof}
    Follows immediately from \autoref{linsys:pullbacks_linsys:claim:criterion} and \autoref{linsys:pullbacks_div:claim:immersion_closure}
\end{proof}

\begin{claim}\label{linsys:pullbacks_ls:claim:lin_comb_of_secs}
    Let $f: X' \to X$ be a morphism of non-empty normal noetherian integral schemes, $\mathcal L$ be an invertible sheaf, $s_1, \dots, s_n \in \Gamma(X, \mathcal L)$ be sections and $s \defeq~\lambda_1 s_1 +~\hdots + \lambda_n s_n$, where $\lambda_i \in k$. If $\div_\mathcal L s$ lifts along $f$, then $\div_\mathcal L s_i$ lifts along $f$ for some $i$.
\end{claim}
\begin{proof}
    Let $\eta$ be the image of the generic point of $X'$. Without loss of generalty $X = \Spec \O_{X, \eta}$ --- then we can assume that $\mathcal L = \O_X$ and $s_1, \dots, s_n, s \in \O_{X, \eta}$. Now, if $\div_\mathcal L s = \div s$ lifts along $f$, then $s \not \in \mathfrak m_\eta$, so there must be $i$ such that $s_i \not \in \mathfrak m_\eta$ and $\div s_i$ lifts along $f$.
\end{proof}

\subsection{Equivariant Cartier Divisors and Support functions}\label{linsys:ssec:equivar_div}

\begin{claim}
    Let $X_\Sigma$ be a toric variety. Fix an equivariant Cartier divisor $D \in~\CaDiv_{T^n} X_\Sigma$ and $\lambda: T^1 \to T^n$ such that $\lambda \in |\Sigma|$, i.e. it admits a unique extensions $\bar \lambda: \A^1 \to X_\Sigma$. Then $D$ lifts along $\bar \lambda$. 
\end{claim}
\begin{proof}
    Clearly $\Supp D \cap T^n = \varnothing$. But $\bar \lambda (T^1) \subset T^n$, so $\bar \lambda(\A^1) \not \subset \Supp D$, hence by \autoref{linsys:pullbacks_div:thm:equiv} we are done.
\end{proof}

\begin{definition}
    Let $X_\Sigma$ be a toric variety and $D \in \CaDiv_{T^n} X_\Sigma$ be an equivariant Cartier divisor. Then we define the support function of $D$ as $\psi_D: |\Sigma| \to \Z$, $\lambda \mapsto~\deg \bar \lambda^* D$, where $\bar \lambda: \A^1 \to X_\Sigma$ is the unique extension of $\lambda: T^1 \to T^n$ (cf. claim above).
\end{definition}

\begin{example}
    For $\chi \in M$ we have $\psi_{\div \chi} = \langle \chi, - \rangle|_{|\Sigma|}$.
\end{example}

\begin{definition}
    Let $\Sigma$ be a fan in $N$. We say that a function $\psi: |\Sigma| \to \Z$ is \textbf{linear on cones} if for any $\sigma \in \Sigma$ and any two points $x, y \in \sigma$ we have $\psi(x + y) = \psi(x) + \psi(y)$.
\end{definition}

\begin{remark}
    Any linear combination with integer coefficients of linear on cones functions $|\Sigma| \to \Z$ is again a linear on cones function, so linear on cones functions $|\Sigma| \to \Z$ is an abelian group.
\end{remark}

\begin{theorem}\label{linsys:equivar_div:thm:homomorphism}
    Let $X_\Sigma$ be a toric variety. Then $D \mapsto \psi_D$ is an isomorphism from equiavariant Cartier divisors $\CaDiv_{T^n} X_\Sigma$ onto the abelian group of linear on cones functions $|\Sigma| \to \Z$.
\end{theorem}
\begin{proof}
    \cite[Th. 4.2.12]{CoxLittleSchenk} (note that in our definition the support function has the opposite sign).
\end{proof}

\begin{remark}
     Let $X_\Sigma$ be a toric variety and $D \in \CaDiv_{T^n} X_\Sigma$ be an equivariant Cartier divisor on $X_\Sigma$. Denote by $\Sigma(1)$ the set of primitive vectors of rays of $\Sigma$. Recall that we have the canonical correspondence $v \mapsto D_v$ between $\Sigma(1)$ and prime equivariant divisors. Then $D = \sum_{v \in \Sigma(1)} \psi_D(v) \cdot D_v$.
\end{remark}
\begin{proof}
    Obvious from local computation.
\end{proof}

\begin{claim}\label{linsys:equivar_div:claim:eff_criterion}
    Let $X_\Sigma$ be a toric variety and $D \in \CaDiv_{T^n} X_\Sigma$ be an equivariant Cartier divisor on $X_\Sigma$. Then $D$ is effective if and only if $\psi_D(|\Sigma|) \subset \Z_{\ge 0}$.
\end{claim}
\begin{proof}
    Clearly if $D$ is effective, then $\psi_D(|\Sigma|) \subset \Z_{\ge 0}$ as the pullback of an effective divisor is effective. The converse: using the notation from the above remark we have that $D = \sum_{v \in \Sigma(1)}\psi_D(v) \cdot D_v$. Since $\psi_D(v) \ge 0$ for all $v \in \Sigma(1)$ we get that $D$ is effective.
\end{proof}

\begin{claim}\label{linsys:equivar_div:claim:global_secs}
    Let $X_\Sigma$ be a toric variety and $D \in \CaDiv_{T^n} X_\Sigma$ be an equivariant Cartier divisor on $X_\Sigma$. Then we have that\footnote{recall that for a set $S$ we denote by $k \cdot S$ the vector space spanned by $S$.} 
    \[
    \Gamma(X_\Sigma, \O(D)) = k \cdot \{ \chi \in M\ |\ \langle\chi, - \rangle|_{|\Sigma|} + \psi_D \ge 0 \}.
    \]
\end{claim}
\begin{proof}
    As $\Gamma(X_\Sigma, \O(D))$ is a $T^n$-invariant subspace of $R(X_\Sigma)$, we get that 
    \[
    \Gamma(X_\Sigma, \O(D)) = k \cdot \left(M \cap \Gamma(X_\Sigma, \O(D)\right)
    \]
    by Artin's lemma on linear independence of characters. By definition, $\chi \in \Gamma(X_\Sigma, \O(D))$ if and only if $\div \chi + D \ge 0$. By \autoref{linsys:equivar_div:claim:eff_criterion} it is equivalent to $\psi_{D + \div \chi} \ge 0$. By \autoref{linsys:equivar_div:thm:homomorphism} $\psi_{D + \div \chi} = \psi_D + \langle \chi, - \rangle|_{\Sigma}$, so we are done.
\end{proof}

\begin{claim}\label{linsys:equivar_div:claim:orbit_lifting}
    Let $X_\Sigma$ be a toric variety and $D \in \CaDiv_{T^n} X_\Sigma$ be an equivariant Cartier divisor on $X_\Sigma$. Fix $\sigma \in \Sigma$ and denote by $\O_\sigma \subset X_\Sigma$ the corresponding orbit. Then the following are equivalent:
    \begin{enumerate}[1)]
        \item $D$ lifts along $\O_\sigma \xhookrightarrow{} X_\Sigma$;
        \item $D$ lifts along $\overline{\O_\sigma} \xhookrightarrow{} X_\Sigma$;
        \item $\psi_D(\sigma) = 0$.
    \end{enumerate}
\end{claim}
\begin{proof}
    \begin{description}
        \item[$1) \iff 2)$] immediate from  \autoref{linsys:pullbacks_div:claim:immersion_closure}.
        \item[$1) \iff 3)$] $\O_\sigma \subset X_\sigma$, so without loss of generality $X_\Sigma = X_\sigma$. In particular, $\psi_D(\sigma) =~0$ is equivalent to $D =0$. Also now $\O_\sigma$ is the smallest orbit in $X_\sigma$ and it is contained in all non-empty closed equivariant subsets of $X_\Sigma$ as it is contained in the closure of any orbit. $\Supp D$ is an equivariant closed subset of $X_\Sigma$, so $D$ lifts along $\O_\sigma \xhookrightarrow{} X_\sigma$ if and only if $\Supp D = \varnothing$, i.e. $D = 0$.
    \end{description}
\end{proof}

\begin{claim}\label{linsys:equivar_div:claim:pullback_along_toric_morphism}
    Let $\phi: \Sigma \to \Sigma'$ be a morphism of fans\footnote{i.e. a homomorphism of lattices $\phi: N \to N'$ such that for any $\sigma \in \Sigma$ there is $\sigma' \in \Sigma'$ so that $\phi(\sigma) \subset \sigma'$.} and $\Phi: X_\Sigma \to X_{\Sigma'}$ be the corresponding toric morphism of toric varieties. Then any equivariant Cartier divisor $D \in \CaDiv_{T^n} X_{\Sigma'}$ lifts along $\Phi$. Moreover, $\psi_{\Phi^* D} = \psi_D \circ \phi$.
\end{claim}
\begin{proof}
    Denote by $\O \subset X_\Sigma$, $\O' \subset X_{\Sigma'}$ the open orbits. Then $\Supp D \cap \O' = \varnothing$ and $\Phi(\O) \subset \O'$, so $\Phi(X_\Sigma) \not \subset \Supp D$, hence by \autoref{linsys:pullbacks_div:thm:equiv} $D$ lifts along $\Phi$. Now, fix any $\lambda \in |\Sigma|$ and let $\bar \lambda: \A^1 \to X_\Sigma$ be the corresponding morphism. Clearly $\Phi \circ \bar \lambda: \A^1 \to X_{\Sigma'}$ is the same as the morphism $\overline{\phi(\lambda)}$. So, from \autoref{linsys:pullbacks_div:claim:functorial_non_eff} we get that:
    \[
    \psi_{\Phi^*D} (\lambda) = \deg \bar \lambda^* (\Phi^*D) = \deg (\Phi \circ \bar \lambda)^* D = \deg \overline{\phi(\lambda)}^* D = \psi_D(\phi(\lambda)).
    \]
\end{proof}

\begin{remark}
    Let $X_\Sigma$ be a toric a variety and $\tau \in \Sigma$ be a cone. Then $\overline{\O_\tau}$ is a toric variety as well and we denote its fan by $\Sigma/\tau$; we warn the reader that our notation is not standard and this fan is also denoted by $Star(\tau)$, cf. \cite[Par. 3.2]{CoxLittleSchenk}.
\end{remark}

\begin{definition}
    Let $\Sigma$ be a fan, $\tau \in \Sigma$ be a cone and $\psi: |\Sigma| \to \Z$ be a function that is linear on cones such that $\psi(\tau) = 0$. Then for any $\sigma \in \Sigma$ such that $\tau \subset \sigma$ we have that $\psi|_\sigma: \sigma \to \Z$ factors uniquely through $\pi_\sigma: \sigma \to \sigma/\tau$ --- clearly all these factorizations are preserved when we intersect cones containing $\tau$, so there must be a unique function $\psi/\tau: |\Sigma/\tau| \to \Z$ such that $\psi|_\sigma = \psi/\tau \circ \pi_\sigma$ for any $\sigma \in \Sigma$ containing~$\tau$.
\end{definition}

\begin{claim}\label{linsys:equivar_div:claim:character_orbit_restriction}
    Let $X_\Sigma$ be a toric variety and $\tau \in \Sigma$ be a cone and denote by $M_\tau$ the character lattice of $\overline \O_\tau$. Then the natural isomorphism $\tau^\bot \to M_\tau$ coincides with $\chi \mapsto \chi|_{\overline \O_\tau}$.
\end{claim}
\begin{proof}
    We will do it by induction on $\dim \tau$. Base: $\dim \tau = 1$. Let $\chi \in M$ be a character such that $\chi\in \tau^\bot$. Since $\O_\tau$ is dense in $\overline \O_\tau$ we can check that two the restriction agree only on $\O_\tau$, so we could assume without loss of generality that $X_\Sigma = X_\tau$. Then $X_\Sigma \simeq T^{n - 1} \times \A^1$ and we get that the two restrictions agree by direct computation.  Now, the inductive step. There must be $\tau' \subset \tau$ such that $\dim \tau' = 1$. Restrictions to $\overline{\O_{\tau'}}$ agree by the base of induction. $\dim \tau/\tau' = \dim \tau - 1$, so restrictions to $\O_{\tau/\tau'} = \O_{\tau}$ from $\overline{\O_{\tau'}}$ agree by inductive hypothesis.
\end{proof}

\begin{claim}\label{linsys:equivar_div:claim:supp_func_orbit}
    Let $\Sigma$ be a fan, $D \in \CaDiv_{T^n} X_\Sigma$ be an equivariant Cartier divisor and $\tau \in \Sigma$ be a cone such that $D$ lifts along $\overline \O_\tau \to X_\Sigma$ (i.e. $\psi_D(\tau) = 0$). Recall that we have that $\overline \O_\tau \cong X_{\Sigma/\tau}$. Then $\psi_{D|_{\overline \O_\tau}} = \psi_D/\tau$.
\end{claim}
\begin{proof}
    Denote $D_\tau \defeq D|_{\overline{\O_\tau}}$. Fix any $\sigma \in \Sigma$ such that $\tau \subset \sigma$ and denote by $\pi: \sigma \to \sigma/\tau$ the canonical projection. We want to show that $\psi_D = \psi_{D_\tau}|_{\sigma/\tau} \circ \pi$ --- since $\pi$ is an epimorphism, it would prove that $\psi_{D_\tau}|_{\sigma/\tau} = (\psi_D/\tau)|_{\sigma/\tau}$. Without loss of generality $X_\Sigma = X_\sigma$. Then $D = \div \chi$. Since $\psi_D(\tau) = 0$, we get that $\chi \in \tau^\bot$ -- i.e. $\psi_D: \sigma \to \Z$ is a linear function that is identically zero on $\tau$, hence $\psi_D = \langle \chi, - \rangle \circ \pi$, where $\chi$ is considered as a linear map $\sigma/\tau \to \Z$. It remains to show that restricting characters as covectors and restricting characters as rational functions to a closed subvariety is the same thing which is exactly \autoref{linsys:equivar_div:claim:character_orbit_restriction}. 
\end{proof}

\subsection{Equivariant Linear Systems and Combinatorial Data}\label{linsys:ssec:equivar_ls}

\begin{definition}
    Let $\Sigma$ be a fan in $N$, $A \subset M$ be a finite subset, $\psi: |\Sigma| \to \Z$ be a function that is linear on cones. We call the pair $(A, \psi)$ \textbf{a linear system combinatorial datum} if the condition $A|_{|\Sigma|} + \psi \ge 0$ is satisfied, i.e. if for any $\lambda \in |\Sigma|$ and any $\chi \in A$ we have that $\langle \chi, \lambda \rangle + \psi(\lambda) \ge 0$.
\end{definition}

\begin{definition}
    Let $\Sigma$ be a fan and $(A, \psi)$ be a linear system combinatorial datum. We define the equivariant linear system $\d(A, \psi)$ as follows. By \autoref{linsys:equivar_div:thm:homomorphism} there is a unique equivariant Cartier divisor $D \in \CaDiv_{T^n} X_\Sigma$ such that $\psi = \psi_D$. By \autoref{linsys:equivar_div:claim:global_secs} $A \subset \Gamma(X_\Sigma, \O(D))$. So, we define the equivariant linear system as follows: $\d(A, \psi) \defeq (k \cdot A, \O(D))$.
\end{definition}

\begin{theorem}\label{linsys:equivar_ls:thm:classification}
    Let $\Sigma$ be a fan. Then for any equivariant linear system $\d$ on $X_\Sigma$ there is a linear system combinatorial datum $(A, \psi)$ such that $\d \simeq \d(A, \psi)$. Moreover, if $(A, \psi)$, $(A', \psi')$ are linear system combinatorial data, then $\d(A, \psi) \simeq \d(A', \psi')$ if and only if there is $\chi \in M$ such that $A = \chi \cdot A'$, $\psi = \psi' - \langle \chi, - \rangle|_{|\Sigma|}$.
\end{theorem}
\begin{proof}
    Fix an equivariant linear system $\d = (V, \mathcal L)$. Since\footnote{because $\Pic T^n = 0$ and any divisor supported on $X_\Sigma\backslash T^n$ is equivariant} $\Pic X_\Sigma = \Pic_{T^n} X_\Sigma$, there is an equivariant Cartier divisor $D \in \CaDiv_{T^n} X_\Sigma$  such that $\mathcal L \simeq \O(D)$, so we could assume without loss of generality that $\mathcal L = \O(D)$. Now we get the natural embedding $V \subset \Gamma(X, \O(D)) \subset k[T^n] = k \cdot M$, because $D|_{T^n} = 0$. Define $A \defeq V \cap M$ --- we will show that $V = k \cdot A$. To do that, we could replace $k$ with arbitrary field extension, in particular we can make $k$ infinite. For infinite $k$ the equality follows from Artin's lemma on characters. Thus, we proved $\d \simeq \d(A, \psi_D)$.

    Now we prove that if $A = \chi\cdot A'$ and $\psi = \psi' - \langle \chi, - \rangle$, then $\d(A, \psi) \simeq \d(A', \psi')$. Let $D, D' \in \CaDiv_{T^n}$ be such that $\psi = \psi_D$, $\psi' = \psi_{D'}$ --- then $\d(A, \psi) = (k \cdot A, \O(D))$, $\d(A', \psi') = (k \cdot A', \O(D'))$. By \autoref{linsys:equivar_div:thm:homomorphism} we have that $D = D' - \div \chi$, so we can define the morphism $\O(D) \to \O(D')$, $f \mapsto \chi \cdot f$ which is clearly an isomorphism (its inverse is $\chi^{-1} \cdot -$) that maps $k \cdot A$ onto $k \cdot A'$, Hence $\d(A, \psi) \simeq \d(A', \psi')$. 

    Finally, assume that $\d(A, \psi) \simeq \d(A', \psi')$. Let $D, D' \in \CaDiv_{T^n}$ be such that $\psi = \psi_D$, $\psi' = \psi_{D'}$. In particular, we get that $\O(D) \simeq \O(D')$, so by \autoref{prelim:div_lin_sys:claim:divisorial_eq} $D - D' = \div \chi$ for some\footnote{a priori we only know that $D - D' = \div f$, but $\div f$ must be equivariant, so $f$ must be proportional to a character} $\chi \in M$. Replacing $D'$ with $D + \langle\chi, - \rangle$ and $A'$ with $\chi^{-1} \cdot A$ we can assume that $\psi = \psi'$ and $D = D'$. Since $\O(D)$ is of rank 1 and torsion-free, any automorphism $\O(D) \to \O(D)$ must come from\footnote{from torsionless we get that any automorphism is determined by its germ at the generic point, but $\O(D)_\eta \simeq k(X_\Sigma)$, where $\eta$ is the generic point, and all automorphisms of $k(X)$ are multiplications by elements of $k(X_\Sigma)^\times$. If $f \in k(X_\Sigma)$ is such that $f \cdot \O(D) \subset \O(D)$, $f^{-1} \cdot \O(D) \subset \O(D)$, then $f \in \O_{X_\Sigma, \xi}^\times$ for any $\xi \in X_\Sigma$ such that $\dim \O_{X, \xi} = 1$, so by Hartogs' lemma $f, f^{-1}$ are regular on $X_\Sigma$.} a multiplication by an invertible regular function that is regular on $X_\Sigma$. All regular invertible functions on $X_\Sigma$ are of the form $\lambda \mu$, where $\lambda \in k^\times$, $\mu \in M \cap |\Sigma|^\bot$. So, the map $k\cdot A \to k \cdot A'$ induced by our isomorphism $\d(A, \psi) \simeq \d(A', \psi)$ must be of the form $f \mapsto \lambda \mu \cdot f$. Hence, $\mu \cdot A = A'$. As $\psi + \langle \mu, - \rangle|_{|\Sigma|} = \psi$ (because $\mu \in |\Sigma|^\bot$), we are done.
\end{proof}

\begin{corollary}\label{linsys:equivar_ls:cor:ls_contains_equivar_div}
    If $\d$ is an equivariant linear system on a toric variety $X_\Sigma$, then there is an equivariant divisor $D \in \d$.
\end{corollary}
\begin{proof}
    By \autoref{linsys:equivar_ls:thm:classification} $\d = \d(A, \psi)$ for some combinatorial datum $(A, \psi)$. Let $D_0 \in \CaDiv_{T^n} X_\Sigma$ be such that $\psi = \psi_{D_0}$. Fix any $\chi \in A$ --- then $D \defeq D_0 + \div \chi \in \d$ is equivariant.
\end{proof}

\begin{claim}
    Let $\phi: \Sigma \to \Sigma'$ be a morphism of fans and $\Phi: X_\Sigma \to X_{\Sigma'}$ be the corresponding toric morphism of toric varieties. Then any equivariant linear system $\d$ on $X_\Sigma'$ lifts along $\Phi$. Moreover, if $\d = \d(A, \psi)$, then $\Phi^* \d = \d(\phi^*(A), \psi \circ \phi)$, where $\phi^*: M' \to M$ is the morphism on character lattices induced by the fans morphism $\phi$.
\end{claim}
\begin{proof}
    By \autoref{linsys:equivar_ls:cor:ls_contains_equivar_div} there is an equivariant divisor $D \in \d$. By \autoref{linsys:equivar_div:claim:pullback_along_toric_morphism} $D$ lifts along $\Phi$. Hence, by \autoref{linsys:pullbacks_linsys:claim:criterion} $\d$ lifts along $\Phi$.

    Now, let $\d = \d(A, \psi)$. Let $D \in \CaDiv_{T^n} X_\Sigma$ be such that $\psi = \psi_D$. Then we have $\d =(k \cdot A, \O(D))$. By \autoref{linsys:pullbacks_div:rem:pullback_commutes_w_O} $\Phi^*(\O(D)) \cong \O(\Phi^*D)$ and by \autoref{linsys:equivar_div:claim:pullback_along_toric_morphism} $\psi_{\Phi^* D} =~\psi_D \circ \phi = \psi \circ \phi$. Obviously the image of $A$ under $\Gamma(X_{\Sigma'}, \O(D)) \to~\Gamma(X_\Sigma, \O(\Phi^*D))$ is $\phi^*(A)$. Thus, $\Phi^* \d = \d(\phi^*(A), \psi \circ \phi)$.
\end{proof}

\begin{claim}\label{linsys:equivar_ls:claim:lifting_criterion}
    Let $X_\Sigma$ be a toric variety and $\d = \d(A, \psi)$ be an equivariant linear system on $X_\Sigma$, $\sigma \in \Sigma$ be a cone. Then the following are equivalent:
    \begin{enumerate}[1)]
        \item $\d$ lifts along $\O_\sigma \xhookrightarrow{} X_\Sigma$;
        \item $\d$ lifts along $\overline{\O_\sigma} \xhookrightarrow{} X_\Sigma$;
        \item There is $\chi \in A$ such that $(\langle \chi, - \rangle + \psi_D)|_\sigma \equiv 0$.
    \end{enumerate}
\end{claim}
\begin{proof}
    \begin{description}
        \item[$1) \iff 2)$] \autoref{linsys:pullbacks_linsys:claim:immersion_closure}
        \item[$3) \implies 1)$] If there is such $\chi$, then denote by $D \in \CaDiv_{T^n}^+ X_\Sigma$ the effective equivariant Cartier divisor $D$ such that $\psi_D = \psi + \langle \chi, - \rangle|_{|\Sigma|}$ --- obviously $D \in \d$. By \autoref{linsys:equivar_div:claim:orbit_lifting} $D$ lifts along $\O_\sigma \xhookrightarrow{} X_\Sigma$, so by \autoref{linsys:pullbacks_linsys:claim:criterion} $\d$ lifts along $\O_\sigma \xhookrightarrow{} X_\Sigma$.
        \item[$1) \implies 3)$] Let $D \in \CaDiv_{T^n} X_\Sigma$ be such that $\psi = \psi_{D}$. Then $\d = (k \cdot A, \O(D)))$. By \autoref{linsys:pullbacks_linsys:claim:criterion} there is $D_0 \in \d$ that lifts along $\O_\sigma \xhookrightarrow{} X_\Sigma$, so there $s \in k \cdot A$ such that $D_0 = \div_{\O(D)} s$ and $\div_{\O(D)} s$ lifts along $\O_\sigma \xhookrightarrow{} X_\Sigma$. By \autoref{linsys:pullbacks_ls:claim:lin_comb_of_secs} there is $\chi \in A$ such that $\div_{\O(D)} \chi = D + \div \chi$ lifts along $\O_\sigma \xhookrightarrow{} X_\Sigma$. By \autoref{linsys:equivar_div:claim:orbit_lifting} $(\psi_D + \langle \chi, - \rangle)|_\sigma \equiv 0$. Since $\psi_D = \psi$, we are done.
    \end{description}
\end{proof}

\begin{claim}
    Let $X_\Sigma$ be a toric variety and $\d = \d(A, \psi)$ be an equivariant linear system on $X_\Sigma$, $\tau \in \Sigma$ be a cone such that $\psi(\tau) = 0$. Then $\d|_{\overline \O_\tau} \simeq \d(A \cap \tau^\bot, \psi/\tau)$.
\end{claim}
\begin{proof}
    Let $D \in \CaDiv_{T^n} X_\Sigma$ be such that $\psi = \psi_D$. Then $\d = (k \cdot A, \O(D))$. By \autoref{linsys:pullbacks_div:rem:pullback_commutes_w_O} $\O(D)|_{\overline \O_\tau} \cong \O(D|_{\overline \O_\tau})$ and by \autoref{linsys:equivar_div:claim:supp_func_orbit} $\psi_{D|_{\overline \O_\tau}} = \psi/\tau$. By \autoref{linsys:equivar_div:claim:character_orbit_restriction} we have $A \cap \tau^\bot \subset \O(\overline \O_\tau, \O(D|_{\overline{\O_\tau}}))$. It remains to show that the image of $A \backslash \tau^\bot$ under $\Gamma(X_\Sigma, \O(D)) \to \Gamma(\overline \O_\tau, \O(D)|_{\overline \O_\tau})$ is zero. For any $\chi \in A \backslash \tau^\bot$ we have that $\overline \O_\tau \subset \Supp \div_{\O(D)} \chi$ by \autoref{linsys:equivar_div:claim:orbit_lifting} and \autoref{linsys:pullbacks_div:thm:equiv}, so from \autoref{prelim:div_lin_sys:claim:sec_vanishing_criterion} we have that the image of $\chi$ in $\Gamma(\overline \O_\tau, \O(D))$ is 0.
\end{proof}

\begin{corollary}
    Let $X_\Sigma$ be a toric variety and $\d = \d(A, \psi)$ be an equivariant linear system on $X_\Sigma$, $\tau \in \Sigma$ be a cone and $\chi \in A$ be such that $(\langle \chi, - \rangle + \psi_D)|_\tau \equiv 0$. Then $\d$ lifts along $\overline \O_\tau \xhookrightarrow{} X_\Sigma$ and $\d|_{\overline\O_\tau} = \d\left((\chi^{-1} \cdot A) \cap \tau^\bot, \left(\psi + \langle \chi, - \rangle\right)/\tau\right)$.
\end{corollary}

\begin{corollary}\label{linsys:equivar_ls:cor:restriction_to_orbit}
     Let $X_\Sigma$ be a toric variety and $\d = \d(A, \psi)$ be an equivariant linear system on $X_\Sigma$, $\tau \in \Sigma$ be a cone and $\chi \in A$ be such that $(\langle \chi, - \rangle + \psi_D)|_\tau \equiv 0$. Then $\d$ lifts along $\O_\tau \xhookrightarrow{} X_\Sigma$ and $\d|_{\O_\tau} = \d\left((\chi^{-1} \cdot A) \cap \tau^\bot, 0\right)$.
\end{corollary}
\section{Counting Irreducible Components}\label{sec:counting}

\subsection{Preparations}

\begin{lemma}
    Let $X$ be a noetherian scheme, $Z \subset X$ be a closed subset such that $\dim Z  \ge \dim (X \backslash Z)$. Then $Z$ contains an irreducible component of $X$.
\end{lemma}
\begin{proof}
    Assume the contrary. Then the generic points of irreducible components of $X$ lie in $X \backslash Z$, so $\overline{X \backslash Z} = X$, hence $Z \subset \big(\overline{X \backslash Z}\big) \backslash (X \backslash Z)$. For any subset $U \subset X$ we have that $\dim \overline U \backslash U < \dim U$, so $\dim Z < \dim (X \backslash Z)$ which is a contradiction.
\end{proof}
\begin{corollary}\label{count:preps:cor:irred_comp_cond}
    If $X$ is a noetherian scheme and $Z \subset X$ is a closed irreducible subset such that $\dim Z \ge \dim (X \backslash Z)$, then $Z$ is an irreducible component of $X$.
\end{corollary}

\begin{lemma}\label{count:preps:lem:intersec_dim}
    Let $\mathfrak d_1, \dots, \mathfrak d_m$ be non-empty equivariant linear systems on $T^n$, $n \le m$. Then one of the following is satisfied:
    \begin{enumerate}
        \item $D_1 \cap \hdots \cap D_m = \varnothing$ for the general $D_1 \in \mathfrak d_1, \dots, D_m \in \mathfrak d_m$;
        
        \item $D_1 \cap \hdots \cap D_m$ is of pure dimension $n - m$ for the general $D_1 \in \mathfrak d_1, \dots, D_m \in \mathfrak d_m$.
    \end{enumerate}
\end{lemma}
\begin{proof}
    By \autoref{linsys:equivar_ls:thm:classification} there are finite subsets $A_i \subset M$ such that $\d_i = \d(A_i, 0)$ (we use that $\Pic T^n = 0$). Consider the space\footnote{for a vector space $V$ by $\A(V)$ we mean the corresponding affine space (i.e. the corresponding variety)}  
    \[
    X \defeq \bigg\{ (f_1, \dots, f_m, p) \in \A(k \cdot A_1) \times \dots \times \A(k \cdot A_m) \times T^n\ \bigg|\ f_1(p) = \dots = f_m(p) = 0 \bigg\}
    \]
    By \cite[Cor. 2.14]{Zhizhin2024} $\dim X = n - m + \sum_i |A_i| = n - m + \dim \prod_i \A(k \cdot A_i)$. Assume that the intersection $D_1 \cap \dots \cap D_m$ is non-empty for the general $D_i \in \d_i$. It implies that the general fibre of the projection $X \to \prod_i \A(k \cdot A_i)$ is non-empty, i.e. the projection is domiminant. Then the dimension of the general fibre is the relative dimension of the projection, i.e. $n - m$. Hence, the dimension of $D_1 \cap \dots \cap D_m$ is equal to $n - m$ for the general $D_i \in \d_i$. It is an intersection of $m$ effective Cartier divisors, so the dimension is pure.
\end{proof}

\subsection{Counting Theorem}

\begin{definition}
    Let $\d_1, \dots, \d_m$ be equivariant linear systems on a toric variety $X_\Sigma$. We denote by $K_{\O_\sigma}(\d_1|_{\O_\sigma}, \dots, \d_m|_{\O_\sigma})$ the number of geometric irreducible components of the intersection $D_1 \cap \dots \cap D_m \cap \O_\sigma$ for the general $D_1 \in \d_1, \dots, D_m \in \d_m$.
\end{definition}

\begin{remark}
    From \autoref{linsys:equivar_ls:thm:classification} we know that there are $A_i \subset M$, $\psi_i: |\Sigma| \to \Z$ such that $\d_i = \d(A_i, \psi_i)$ and from \autoref{linsys:equivar_ls:claim:lifting_criterion} we know that $\d_i$ degenerates along $\O_\sigma \to X_\Sigma$ if and only if for any $\chi \in A_i$ we have $(\langle \chi, - \rangle + \psi_i)|_\sigma \not \equiv 0$. Otherwise by \autoref{linsys:equivar_ls:cor:restriction_to_orbit} $\d_i|_{\O_\sigma} = \d\left((\chi^{-1} \cdot A_i) \cap \sigma^\bot, \left(\psi_i + \langle \chi, - \rangle\right)/\tau\right)$. So one could compute $K_{\O_\sigma}(\d_1|_{\O_\sigma}, \dots, \d_m|_{\O_\sigma})$ using \autoref{prelim:ssec:khovanskii} via replacing $\d_i|_{\O_\sigma}$ by $(A_i \cdot \chi^{-1}_{i, \sigma}) \cap \sigma^\bot$ for appropriate $\chi_{i, \sigma}$. 
\end{remark}

\begin{theorem}
    Let $\mathfrak d_1, \dots, \mathfrak d_m$ be equivariant linear systems on a toric variety $X_\Sigma$. Consider the sets $
    \mathcal D(\sigma) \defeq \{ i\ |\ \d_i \text{ degenerates along } \O_\sigma \xhookrightarrow{} X_\Sigma \}$
    and the function $d: \Sigma \to \mathbb Z_{\ge 0}$, $\sigma \mapsto |\mathcal D(\sigma)| - \dim \sigma$. Define the set
    \[
    \mathcal S = \{ \sigma \in \Sigma \ |\ d(\sigma) \ge d(\tau) \text{ for all faces } \tau \text{ of } \sigma\}.
    \]
    Then for the general $D_1 \in \mathfrak d_1, \dots, D_m \in \mathfrak d_m$ the intersection $D_1 \cap \dots \cap D_m$ has the following number of irreducible components:
    \[
    \sum_{\sigma \in \mathcal S} K_{\O_\sigma}(\mathfrak d_1|_{\O_\sigma}, \dots, \mathfrak d_m|_{\O_\sigma})
    \]
    where if $\mathfrak d_i$ degenerates along $\O_\sigma \xhookrightarrow{} X_\Sigma$, then we assume\footnote{i.e. we assume $K_\sigma(\mathfrak d_1|_\sigma, \dots, \mathfrak d_m|_\sigma) \defeq K_\sigma (\mathfrak d_{i_1}|_\sigma, \dots, \mathfrak d_{i_l}|_\sigma)$ for $\{ i_1, \dots, i_l \} = \{1, \dots, m\} \backslash \mathcal D(\sigma)$} $\mathfrak d_i|_\sigma = \varnothing$.
\end{theorem}
\begin{proof}
    We prove that the irreducible components of $D_1 \cap \hdots D_m \cap \O_\sigma$ are\footnote{in the sense that the generic points of these components is a subset of the generic points of irreducible components of $D_1 \cap \hdots D_m$} irreducible components of $D_1 \cap \hdots \cap D_m$ if and only if $d(\sigma) \ge d(\tau)$ for all faces $\tau \subset \sigma$ --- then the formula of the theorem is obvious. First fix any $\sigma \in \mathcal S$. From \autoref{count:preps:lem:intersec_dim} we know that for the general $D_1 \in \mathfrak d_1, \dots, D_m \in \mathfrak d_m$ the intersection $D_1 \cap \hdots D_m \cap \O_\sigma$ is of pure dimension $n - m + d(\sigma)$ if non-empty, which is greater or equal than $\dim \Big(D_1 \cap~\hdots \cap D_m \cap (X_\sigma \backslash \O_\sigma)\Big)$ because $X_\sigma$ consists of $X_\sigma \cap \O_\tau$ for faces $\tau \subset \sigma$. Hence, by \autoref{count:preps:cor:irred_comp_cond} all irreducible components of $D_1 \cap \hdots D_m \cap \O_\sigma$ are irreducible components of $D_1 \cap~\hdots D_m \cap~X_\sigma$ and of $D_1 \cap \hdots \cap D_m$ as $X_\sigma$ is open in $X_\Sigma$. 
    
    Now fix any cone $\sigma \in \Sigma$ such that\footnote{i.e. $\sigma \not \in \mathcal S$} there is a face $\tau \subset \sigma$ with $d(\tau) > d(\sigma)$. Since $\O_\sigma \subset \overline{\O_\tau}$, if any irreducible component of $D_1 \cap \hdots D_m \cap \O_\sigma$ is an irreducible component of $D_1 \cap \hdots \cap D_m$, then it is also an irreducible component of $D_1 \cap \hdots \cap D_m \cap \overline{\O_\tau}$. Consider the quotient cone $\Sigma/\tau$ and the corresponding variety $X_{\Sigma/\tau} = \overline{\O_\tau}$. Without loss of generality $\Sigma = \Sigma/\tau$ --- and $d(\tau) = 0$. Then $\dim (D_1 \cap \hdots \cap D_m \cap \O_\sigma) < n - m$ --- it cannot contain no irreducible components as the dimension of the intersection of $m$ effective Cartier divisors on an $n$-dimensional irreducible variety cannot be less than $n - m$ at any point. Thus we are done
\end{proof}

\bibliographystyle{bibstyle}
\bibliography{bibliography}

\begin{thebibliography}{CLS12}
\expandafter\ifx\csname url\endcsname\relax
  \def\url#1{\texttt{#1}}\fi
\expandafter\ifx\csname doi\endcsname\relax
  \def\doi#1{\burlalt{doi:#1}{http://dx.doi.org/#1}}\fi
\expandafter\ifx\csname urlprefix\endcsname\relax\def\urlprefix{URL }\fi
\expandafter\ifx\csname href\endcsname\relax
  \def\href#1#2{#2}\fi
\expandafter\ifx\csname burlalt\endcsname\relax
  \def\burlalt#1#2{\href{#2}{#1}}\fi

\bibitem[B75]{bernstein}
D.~Bernstein.
\newblock The number of roots of a system of equations.
\newblock {\em Functional Analysis and Its Applications}, 9(3):183--185, 1975.
\newblock \doi{https://doi.org/10.1007/BF01075595}.

\bibitem[CLS12]{CoxLittleSchenk}
J.~Hausen.
\newblock David a. cox, john b. little, henry k. schenck: “toric varieties”: American mathematical society, 2011, 841 pp.
\newblock {\em Jahresbericht der Deutschen Mathematiker-Vereinigung}, 114(3):171–175, June 2012.
\newblock \doi{10.1365/s13291-012-0048-9}.

\bibitem[Ew96]{Ewald1996}
G.~Ewald.
\newblock {\em Combinatorial Convexity and Algebraic Geometry}.
\newblock Springer New York, 1996.
\newblock \doi{10.1007/978-1-4612-4044-0}.

\bibitem[HS77]{Hartshorne1977}
R.~Hartshorne.
\newblock {\em Algebraic Geometry}.
\newblock Springer New York, 1977.
\newblock \doi{10.1007/978-1-4757-3849-0}.

\bibitem[K77]{Kushnirenko1977}
A.~G. Kushnirenko.
\newblock Newton polytopes and the bezout theorem.
\newblock {\em Functional Analysis and Its Applications}, 10(3):233–235, 1977.
\newblock \doi{10.1007/bf01075534}.

\bibitem[KH16]{Khovanskii2016}
A.~G. Khovanskii.
\newblock Newton polytopes and irreducible components of complete intersections.
\newblock {\em Izvestiya: Mathematics}, 80(1):263--284, Feb. 2016.
\newblock \doi{10.1070/im8307}.

\bibitem[Stacks]{stacks-project}
T.~{Stacks Project Authors}.
\newblock \textit{Stacks Project}.
\newblock \url{https://stacks.math.columbia.edu}.

\bibitem[Zhi24]{Zhizhin2024}
A.~Zhizhin.
\newblock Irreducibility of toric complete intersections, 2024.
\newblock \doi{10.48550/ARXIV.2409.00188}.

\end{thebibliography}

\end{document}